\documentclass[12pt]{amsart}
\usepackage{epsfig}
\usepackage{graphics}
\usepackage{color}
\usepackage{tikz}
\usepackage[utf8]{inputenc}
\usepackage{amsmath}
\usepackage{amsfonts}
\usepackage{amssymb}
\usepackage[mathscr]{euscript}
\usepackage{pinlabel}

\usepackage{comment}

\newtheorem{theorem}{Theorem}

\newtheorem{corollary}[theorem]{Corollary}

\newtheorem{lemma}[theorem]{Lemma}
\newtheorem{proposition}[theorem]{Proposition}

\theoremstyle{definition}

\theoremstyle{remark}

 \def\R{{\mathbb{R}}}
 \def\Z{{\mathbb{Z}}}

\def\mod{{\rm Mod}}

 \def\Sp{{\rm Sp}}

\begin{document}

\newenvironment{prooff}{\medskip \par \noindent {\it Proof}\ }{\hfill
$\square$ \medskip \par}
    \def\sqr#1#2{{\vcenter{\hrule height.#2pt
        \hbox{\vrule width.#2pt height#1pt \kern#1pt
            \vrule width.#2pt}\hrule height.#2pt}}}
    \def\square{\mathchoice\sqr67\sqr67\sqr{2.1}6\sqr{1.5}6}
\def\pf#1{\medskip \par \noindent {\it #1.}\ }
\def\endpf{\hfill $\square$ \medskip \par}
\def\demo#1{\medskip \par \noindent {\it #1.}\ }
\def\enddemo{\medskip \par}
\def\qed{~\hfill$\square$}

\def \MCG {\mod(\Sigma_g)}
\def \MCGb {\mod(\Sigma_g^m)}

 \title[Mapping class group is generated by two commutators]
{The mapping class group is generated \linebreak by two commutators}

\author[R. \.{I}. Baykur]{R. \.{I}nan\c{c} Baykur}
\address{Department of Mathematics and Statistics, University of Massachusetts, Amherst, MA 01003-9305, USA}
\email{baykur@math.umass.edu}

\author[M.~Korkmaz]{Mustafa Korkmaz}
 \address{Department of Mathematics, Middle East Technical University, 06800
  Ankara, Turkey
  }
 \email{korkmaz@metu.edu.tr}


\begin{abstract}
We show that the mapping class group of any closed connected orientable surface of genus at least five is generated by only two commutators, and if the genus is three or four, by three commutators. 
\end{abstract} 
 \maketitle

\vspace{0.2in}
\section{Introduction}
The mapping class group $\mod(\Sigma_g)$ of a closed connected orientable surface of genus $g$ is known to be perfect, i.e. equal to its commutator subgroup,  when $g \geq 3$ \cite{Powell}. We prove the following peculiar result:

\begin{theorem} \label{thm:1}
The mapping class group $\mod(\Sigma_g)$ is generated by two commutators if $g\geq 5$, and by three commutators if $g \geq 3$. 
\end{theorem}

Our result is clearly sharp when $g \geq 5$, for $\MCG$ is not cyclic. When $g=1$ and $2$, the abelianization of $\mod(\Sigma_g)$ is $\Z_{12}$ and $\Z_{10}$, respectively, so the mapping class group cannot be generated by commutators in these low genera cases.

A geometric implication of Theorem~\ref{thm:1} is that any pair of genus \mbox{$g \geq 5$} surface bundle over the circle are cobordant through a finite sequence of basic building blocks, which are fibrations over two-holed tori, prescribed by the two commutator generators and their inverses.

Since the action of mapping classes on the integral first homology group 
$H_1(\Sigma_g)$ of $\Sigma_g$ induces an epimorphism from $\mod(\Sigma_g)$ 
onto the symplectic group  $\Sp (2g,\Z)$, another immediate implication is the following:

\begin{corollary} 
The symplectic group $\Sp(2g,\Z)$ is generated by two commutators if $g\geq 5$, and by three commutators if $g\geq 3$.
\end{corollary}

Theorem~\ref{thm:1} adds to the ever-growing literature on minimal generating sets for $\mod(\Sigma_g)$; e.g. by $2g+1$ Dehn twists \cite{Dehn, Lickorish, Humphries}, by three involutions \cite{BrendleFarb, KorkmazInvolutions}, or by two general elements \cite{Wajnryb1996, Korkmaz}. 

It is interesting to know for which perfect groups the minimal number of generators is equal to the minimal number of commutator generators. There are numerous other groups satisfying this property. For example; any finite non-abelian simple group, such as the alternating group $A_n$ for $n \geq 5$, is a perfect group generated by two elements, whereas by the resolution of Ore's conjecture \cite{LiebeckEtAl}, any element in such a group is a commutator. The same holds for the special linear group $SL(n, R)$ for various $n \geq 3$ and coefficient rings $R$, which goes back to the classical works of Thompson \cite{Thompson}. However, the situation is much more subtle  for the mapping class group, since $\MCG$, for $g \geq 3$,  is not even uniformly perfect \cite{EndoKotschick}, i.e. there is no fixed positive number that any element in $\mod(\Sigma_g)$ can be expressed as a product of that many commutators, let it be one.

The explicit set of generators we provide for Theorem~\ref{thm:1} consists of a finite order mapping class and an infinite order one (or two) that is a product of disjoint Dehn twists. In Section~\ref{sec:prelim},  we review the basic results regarding Dehn twists. The torsion elements, and their expressions as commutators, come from the symmetries of the surface, and are discussed in Section~\ref{sec:RandT}. Various new generating sets for $\MCG$ featuring the above mapping classes are obtained in  Section~\ref{sec:gen}, and the proof of Theorem~\ref{thm:1} is given in Section~\ref{sec:final}.

\vspace{0.2in}
\noindent \textit{Acknowledgements.} The first author was partially supported by the NSF Grant DMS-$1510395$.  The second author thanks  UMass Amherst for its generous support and wonderful research environment during this project.


\section{Preliminaries}\label{sec:prelim}

The \emph{mapping class group} ${\rm Mod}(\Sigma)$ of a compact connected oriented surface $\Sigma$
is the group of orientation--preserving diffeomorphisms of $\Sigma\to \Sigma$, which restrict to identity near the boundary $\partial \Sigma$, modulo isotopies of the same type. We will be primarily interested in the 
case when $\Sigma= \Sigma_g$, the closed surface of genus $g$.

We denote simple closed curves on $\Sigma$ by lowercase letters such as $a,b,c,d$, and denote \emph{positive (right-handed) Dehn twists} $t_a, t_b, t_c, t_d$ about them by the corresponding capital letters $A,B,C, D$, all with indices. In our notation, both the curves on $\Sigma$ and self-diffeomorphisms of $\Sigma$ should be understood up to isotopy. We use the functional notation for the composition of diffeomorphisms (i.e. for $\phi \psi$, $\psi$ acts on $\Sigma$ first), yet we still express the \emph{commutator of $\phi$ and $\psi$} as $[\phi, \psi]= \phi \psi \phi^{-1} \psi^{-1}$.

We will make repeated use of the following basic relations in ${\rm Mod}(\Sigma)$, without referring to them explicitly: for two simple closed curves 
$a$ and $b$ on $\Sigma$, and for any $f \in {\rm Mod}(\Sigma)$,
\begin{itemize}
\item (\emph{Conjugation}) \ \ \  \ $ft_af^{-1}=t_{f(a)}$,
\item (\emph{Commutativity}) $AB=BA$, if $a$ and $b$ are disjoint,
\item (\emph{Braid relation}) \ \ $ABA= BAB$, if $a$ and $b$ intersect transversely  at one point.
\end{itemize}
\vspace{-0.2cm}
We will also need 
\begin{itemize}
\item (\emph{Lantern relation}) for $x_i, y_j$ the simple closed curves on the four-holed sphere in Figure~\ref{fig:lantern} (embedded in $\Sigma$),  $X_1 X_2 X_3=Y_1 Y_2 Y_3 Y_4.$  
\end{itemize}
\begin{figure}[h!]
 \centering
     \includegraphics[width=3cm]{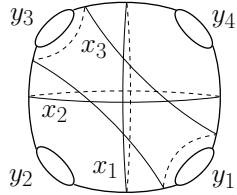}
     \caption{The curves of the lantern relation} \label{fig:lantern}
\end{figure}

All the relations above appeared in the pioneering work of Dehn \cite{Dehn}, who then proved that $\MCG$ is generated by finitely many Dehn twists. Later works of Lickorish~\cite{Lickorish} and Humphries~\cite{Humphries} led to the following minimal collection of Dehn twist generators $\{A_i, B_j, C_k\}$ along the curves $\{a_i, b_j, c_k\}$ on $\Sigma_g$  in Figure ~\ref{fig:models} below.
 
\begin{theorem} {\rm\bf(Dehn-Lickorish-Humphries)}\label{thm:DLK}
The mapping class group $\mod(\Sigma_g)$ is generated by 
$\{ A_1,A_2, B_1,B_2,\ldots ,B_g,C_1,C_2,\ldots ,C_{g-1}\}.$
\end{theorem}


\bigskip
\section{Finite order mapping classes as commutators}\label{sec:RandT}


\begin{figure}[h!]
\begin{tikzpicture}[scale=0.8]
\begin{scope} [xshift=-4cm, yshift=0cm, scale=1.2]
 \draw[very thick] (0,0) circle [radius=2.5cm];
 \draw[very thick, rotate=-25] (0,1.6) circle [radius=0.2cm]; 
 \draw[very thick, rotate=-75] (0,1.6) circle [radius=0.2cm];  
 \draw[very thick, rotate=-125] (0,1.6) circle [radius=0.2cm]; 
 \draw[very thick, rotate=25] (0,1.6) circle [radius=0.2cm]; 
 \draw[very thick, rotate=75] (0,1.6) circle [radius=0.2cm];  
 \draw[very thick, rotate=-175] (0,1.6) circle [radius=0.2cm];  
 \draw[very thick, fill, rotate=143] (0,1.6) circle [radius=0.03cm];  
 \draw[very thick, fill, rotate=130] (0,1.6) circle [radius=0.03cm];  
 \draw[very thick, fill, rotate=117] (0,1.6) circle [radius=0.03cm];  
 \draw[thick, green, rotate=0, rounded corners=4pt]  (-0.48,1.4)--(-0.2, 1.36)--(0.2,1.36)--(0.48,1.4);   
 \draw[thick, green, dashed,  rotate=0, rounded corners=4pt] (-0.48,1.46)--(-0.2, 1.5)--(0.2,1.5)--(0.48,1.46);  
 \draw[thick, green, rotate=-200, rounded corners=4pt]  (-0.48,1.4)--(-0.2, 1.36)--(0.2,1.36);   
 \draw[thick, green, dashed,  rotate=-200, rounded corners=4pt] 
(-0.48,1.46)--(-0.2, 1.5)--(0.2,1.5);  
 \draw[thick, green, rotate=100, rounded corners=4pt] (-0.2, 1.36)--(0.2,1.36)--(0.48,1.4);   
 \draw[thick, green, dashed,  rotate=100, rounded corners=4pt] (-0.2, 1.5)--(0.2,1.5)--(0.48,1.46);  
\draw[thick, green, rotate=50, rounded corners=4pt]  (-0.48,1.4)--(-0.2, 1.36)--(0.2,1.36)--(0.48,1.4);   
 \draw[thick, green, dashed,  rotate=50, rounded corners=4pt] 
 (-0.48,1.46)--(-0.2, 1.5)--(0.2,1.5)--(0.48,1.46);  
\draw[thick, green, rotate=-50, rounded corners=4pt]  (-0.48,1.4)--(-0.2, 1.36)--(0.2,1.36)--(0.48,1.4);   
 \draw[thick, green, dashed,  rotate=-50, rounded corners=4pt] 
 (-0.48,1.46)--(-0.2, 1.5)--(0.2,1.5)--(0.48,1.46);  
\draw[thick, green, rotate=-100, rounded corners=4pt]  (-0.48,1.4)--(-0.2, 1.36)--(0.2,1.36)--(0.48,1.4);   
 \draw[thick, green, dashed,  rotate=-100, rounded corners=4pt] 
 (-0.48,1.46)--(-0.2, 1.5)--(0.2,1.5)--(0.48,1.46);  
\draw[thick, green, rotate=-150, rounded corners=4pt]  (-0.48,1.4)--(-0.2, 1.36)--(0.2,1.36)--(0.48,1.4);   
 \draw[thick, green, dashed,  rotate=-150, rounded corners=4pt] 
 (-0.48,1.46)--(-0.2, 1.5)--(0.2,1.5)--(0.48,1.46);  

 \draw[thick, red, rotate=-25] (0,1.6) circle [radius=0.35cm]; 
 \draw[thick, red, rotate=-75] (0,1.6) circle [radius=0.35cm]; 
 \draw[thick, red, rotate=-125] (0,1.6) circle [radius=0.35cm]; 
 \draw[thick, red, rotate=-175] (0,1.6) circle [radius=0.35cm]; 
 \draw[thick, red, rotate=25] (0,1.6) circle [radius=0.35cm]; 
 \draw[thick, red, rotate=75] (0,1.6) circle [radius=0.35cm]; 
 \draw[thick,blue,  rotate=-25, rounded corners=4pt] (-0.02,1.8)--(-0.1, 2.15)--(-0.02,2.5);  
 \draw[thick,  blue, dashed, rotate=-25, rounded corners=4pt]  (0.02,1.8)--(0.1, 2.15)--(0.02,2.5); 
  \draw[thick, blue,  rotate=-75, rounded corners=4pt] (-0.02,1.8)--(-0.1, 2.15)--(-0.02,2.5);  
 \draw[thick, blue,  dashed, rotate=-75, rounded corners=4pt]  (0.02,1.8)--(0.1, 2.15)--(0.02,2.5); 
 \draw[thick, blue,  rotate=-125, rounded corners=4pt] (-0.02,1.8)--(-0.1, 2.15)--(-0.02,2.5);  
 \draw[thick, blue,  dashed, rotate=-125, rounded corners=4pt]  (0.02,1.8)--(0.1, 2.15)--(0.02,2.5); 
 \draw[thick, blue,  rotate=-175, rounded corners=4pt] (-0.02,1.8)--(-0.1, 2.15)--(-0.02,2.5);  
 \draw[thick, blue,  dashed, rotate=-175, rounded corners=4pt]  (0.02,1.8)--(0.1, 2.15)--(0.02,2.5); 
\draw[thick, blue, rotate=25, rounded corners=4pt] (-0.02,1.8)--(-0.1, 2.15)--(-0.02,2.5);  
 \draw[thick, blue,  dashed, rotate=25, rounded corners=4pt]  (0.02,1.8)--(0.1, 2.15)--(0.02,2.5); 
 \draw[thick, blue,  rotate=75, rounded corners=4pt] (-0.02,1.8)--(-0.1, 2.15)--(-0.02,2.5);  
 \draw[thick,  blue, dashed, rotate=75, rounded corners=4pt]  (0.02,1.8)--(0.1, 2.15)--(0.02,2.5);

 \node[scale=1, blue] at (1.5,2.3) {$a_1$};
 \node[scale=1, blue] at (2.8,0.7) {$a_{2}$};
 \node[scale=1, blue] at ( 2.45,-1.5) {$a_{3}$};
 \node[scale=1, blue] at (0.3,-2.8) {$a_4$};
 \node[scale=1, blue] at (-2.89,0.7) {$a_{g-1}$};
  \node[scale=1, blue] at (-1.3,2.5) {$a_g$};
  
 \node[scale=1, red] at (1.3,1.5) {$b_1$};
\node[scale=1, red] at (2.1,0) {$b_2$};
\node[scale=1, red] at (1.9,-0.7) {$b_3$};
\node[scale=1, red] at (0.7,-1.9) {$b_4$};
\node[scale=1, red] at (-0.75, 0.1) {$b_{g-1}$};
 \node[scale=1, red] at (-1.3,1.5) {$b_g$};

\node[scale=1, green] at (0.75,0.7) {$c_{1}$};
 \node[scale=1, green] at (0.99,-0.2) {$c_{2}$};
  \node[scale=1, green] at (-0.6,0.7) {$c_{g-1}$};
 \node[scale=1, green] at (0,1.1) {$c_g$};
 \node[scale=1, green] at (0.5,-1) {$c_3$};
  \draw[very thick, violet, ->] (0,2.6)--(0,3.3);
 \node[scale=1, violet] at (0.25,3.1) {$z$};
 
\draw[very thick, violet, ->] (2.6,0)--(3.2,0);
 \node[scale=1, violet] at (3,-0.3) {$y$};
 
  \draw[very thick, violet ,rotate=-25] (0,2.6)--(0,3.5);
  \draw[very thick, violet ,rotate=145] (0,2.6)--(0,3.5);
 \node[scale=1, violet] at (1.2,3.05) {$\ell$};

\end{scope}

\begin{scope} [xshift=4cm,scale=0.55]
\filldraw[yellow!20!white] (-2.6,5.41)--(2.61,5.41)--(5.85,1.4)--(4.7,-3.74)-- (0,-6) -- (-4.7, -3.74)--(-5.85,1.4)-- cycle;
\filldraw[red!30!white] (-0.87,1.8)--(0.87,1.8)--(1.96,0.44)-- (1.6,-1.25)--(0,-2) -- (-1.6, -1.25)--(-1.96,0.44)--  cycle;
\draw[very thick] (-4.7,-3.74)-- (0,-6) -- (4.7, -3.74)--(5.85,1.4)  --(2.6,5.41)--(-2.6,5.41)--(-5.85,1.4);
\draw[very thick] (-1.6,-1.25)--(0,-2) -- (1.6, -1.25)--(1.96,0.44)--(0.87,1.8)--(-0.87,1.8)--(-1.96,0.44);
\draw[very thick, dotted](-1.6,-1.25)--(-1.96,0.44);
\draw[very thick, dotted] (-4.7,-3.74)--(-5.85,1.4);
\draw[very thick] (-0.87,1.8) -- (-2.6, 5.41);
\draw[very thick] (0.87,1.8 )--( 2.6*0.66,0.66*5.41);
\draw[very thick]  (1.96,0.44)--(5.85,1.4 );
\draw[very thick]  ( 1.6,-1.25)--(4.7,-3.74 );
\draw[very thick]  ( 0,-2)--(0,-6);
\draw[very thick]  ( -1.6,-1.25)--(-4.7,-3.74 );
\draw[very thick]  (-1.96,0.44)--(-5.85,1.4 );
 \draw[very thick, fill, rotate=100] (-0.5,3.56) circle [radius=0.05cm];  
\draw[very thick, fill, rotate=100] (0,3.6) circle [radius=0.05cm];  
\draw[very thick, fill, rotate=100] (0.5,3.56) circle [radius=0.05cm];  
\draw[very thick, violet, rotate=-25.715] (0,4.05)--(0,7.4);
\draw[very thick, violet, rotate=155.715] (0,3.7)--(0,7.5);
\node[violet, scale=0.9] at (2.5,6.5) {$\ell$};
\draw[thick, red]  (-2, 5)--( 2,5)--(0.7,2.2 )--(-0.7,2.2) --cycle;
\draw[thick, red, rotate=51.43]  (-2, 5)--( 2,5)--(0.7,2.2 )--(-0.7,2.2) --cycle;
\draw[thick, red, rotate=-51.43]  (-2, 5)--( 2,5)--(0.7,2.2 )--(-0.7,2.2) --cycle;
\draw[thick, dotted, red, rotate=102.86]  (-2, 5)--( 2,5)--(0.7,2.2 )--(-0.7,2.2) --cycle;
\draw[thick, red, rotate=-102.86]  (-2, 5)--( 2,5)--(0.7,2.2 )--(-0.7,2.2) --cycle;
\draw[thick, red, rotate=154.29]  (-2, 5)--( 2,5)--(0.7,2.2 )--(-0.7,2.2) --cycle;
\draw[thick, red, rotate=-154.29]  (-2, 5)--( 2,5)--(0.7,2.2 )--(-0.7,2.2) --cycle;
\draw[thick, green, rounded corners=4pt, rotate=0]  ( -0.35,-4)..controls (-0.5,-4.3) and (0.5,-4.3).. (0.35,-4);
\draw[thick, green, dashed, rounded corners=4pt, rotate=0]  ( -0.35,-4)..controls (-0.2,-3.8) and (0.2,-3.8).. (0.35,-4);
\draw[thick, green, rounded corners=4pt, rotate=51.43]  ( -0.35,-4)..controls (-0.5,-4.3) and (0.5,-4.3).. (0.35,-4);
\draw[thick, green, dashed, rounded corners=4pt, rotate=51.43]  ( -0.35,-4)..controls (-0.2,-3.8) and (0.2,-3.8).. (0.35,-4);
\draw[thick, green, rounded corners=4pt, rotate=102.86]  ( -0.35,-4)..controls (-0.5,-4.3) and (0.5,-4.3).. (0.35,-4);
\draw[thick, green, dashed, rounded corners=4pt, rotate=102.86]  ( -0.35,-4)..controls (-0.2,-3.8) and (0.2,-3.8).. (0.35,-4);
%
\draw[thick, blue, rounded corners=4pt, rotate=154.29]  ( -0.35,-4)..controls (-0.5,-4.3) and (0.5,-4.3).. (0.35,-4);
\draw[thick, blue, dashed, rounded corners=4pt, rotate=154.29]  ( -0.35,-4)..controls (-0.2,-3.8) and (0.2,-3.8).. (0.35,-4);
%
\draw[thick, green, rounded corners=4pt, rotate=-51.43]  ( -0.35,-4)..controls (-0.5,-4.3) and (0.5,-4.3).. (0.35,-4);
\draw[thick, green, dashed, rounded corners=4pt, rotate=-51.43]  ( -0.35,-4)..controls (-0.2,-3.8) and (0.2,-3.8).. (0.35,-4);
\draw[thick, green, rounded corners=4pt, rotate=-102.86]  ( -0.35,-4)..controls (-0.5,-4.3) and (0.5,-4.3).. (0.35,-4);
\draw[thick, green, dashed, rounded corners=4pt, rotate=-102.86]  ( -0.35,-4)..controls (-0.2,-3.8) and (0.2,-3.8).. (0.35,-4);
%
\draw[thick, blue, rounded corners=4pt, rotate=-154.29]  ( -0.35,-4)..controls (-0.5,-4.3) and (0.5,-4.3).. (0.35,-4);
\draw[thick, blue, dashed, rounded corners=4pt, rotate=-154.29]  ( -0.35,-4)..controls (-0.2,-3.8) and (0.2,-3.8).. (0.35,-4);
\node[blue, scale=0.9] at (0.85,3.9) {$a_1$};
\node[blue, scale=0.9] at (-0.85 ,3.9) {$a_g$};

\node[blue, scale=0.9] at (0.5,0.7) {$a_2$};
\node[green, scale=0.9] at (-3.6,1.7) {$c_{g-1}$};
\node[green, scale=0.9] at (-3.7,-1.8) {$c_4$};
\node[green, scale=0.9] at (0.9,-4) {$c_3$};
\node[green, scale=0.9] at (3.5, -1.9) {$c_2$};
\node[green, scale=0.9] at (3.8,1.7) {$c_{1}$};

\node[red, scale=0.9] at (-3.2,2.9) {$b_{g}$};
\node[red, scale=0.9] at (4.4,-0.9) {$b_2$};
\node[red, scale=0.9] at (2.2,-3.9) {$b_3$};
\node[red, scale=0.9] at (-2.6,-3.7) {$b_4$};
\node[red, scale=0.9] at (3.6,2.9) {$b_1$};
\node[red, scale=0.9] at (0,2.8) {$b_{g+1}$};
%
\draw[blue, rounded corners=7pt] (2.15,-0.5)..controls(2,-0.4) and (-0.1,2.2)..(0, 2.2);
\draw[blue, dashed, rounded corners=7pt] (2.15,-0.5)..controls(2,0) and (0.1,2.2)
..(0, 2.2);
\draw[very thick, violet, ->] (0,3.7)--(0,6.9);
\node[violet, scale=0.9] at (0.45,6.5) {$z$};

\draw[very thick, violet, ->] (4.1,0)--(6.8,0);
\node[violet, scale=0.9] at (6.4, -0.5) {$y$};


\end{scope}

\begin{scope} [yshift=-6.5cm, xshift=-1cm, scale=0.6]
 \draw[very thick,rounded corners=14pt] (-7.8,-2) -- (10.4,-2)--(11.35,0)-- (10.4,2) -- (-7.8,2) --(-8.75, 0)-- cycle;
 \draw[very thick, xshift=-6.5cm] (0,0) circle [radius=0.6cm];
 \draw[very thick, xshift=-3.9cm] (0,0) circle [radius=0.6cm];
 \draw[very thick, xshift=-1.3cm] (0,0) circle [radius=0.6cm];
\draw[very thick, xshift=6.5cm] (0,0) circle [radius=0.6cm];
 \draw[very thick, xshift=9.1cm] (0,0) circle [radius=0.6cm];
 
\draw[very thick, fill, xshift=1.7cm] (0,0) circle [radius=0.05cm];
\draw[very thick, fill, xshift=2.6cm] (0,0) circle [radius=0.05cm];
\draw[very thick, fill, xshift=3.5cm] (0,0) circle [radius=0.05cm];

 \draw[thick, red, xshift=9.1cm] (0,0) circle [radius=0.9cm];
 \draw[thick, red, xshift=6.5cm] (0,0) circle [radius=0.9cm];
 \draw[thick, red, xshift=-1.3cm] (0,0) circle [radius=0.9cm];
 \draw[thick, red, xshift=-3.9cm] (0,0) circle [radius=0.9cm];
 \draw[thick, red, xshift=-6.5cm] (0,0) circle [radius=0.9cm];

\draw[very thick,rounded corners=14pt] (-7.8,-2) -- (10.4,-2)--(11.35,0)-- (10.4,2) -- (-7.8,2) --(-8.75, 0)-- cycle;
\draw[thick,red, rounded corners=14pt] (-7.6,-1.8) -- (10.2,-1.8)--(11.15,0)-- (10.2,1.8) -- (-7.6,1.8) --(-8.55, 0)-- cycle;

 \draw[thick, green, rounded corners=4pt, xshift=-5.2cm] (-0.7,-0.03) -- (-0.4,-0.23)--(0.4,-0.23) --(0.7,-0.03);
 \draw[thick, green, dashed, rounded corners=4pt, xshift=-5.2cm] (-0.7,0.03) -- (-0.4,0.23)--(0.4,0.23) --(0.7,0.03);
 \draw[thick, green, rounded corners=4pt, xshift=-2.6cm] (-0.7,-0.03) -- (-0.4,-0.23)--(0.4,-0.23) --(0.7,-0.03);
 \draw[thick, green, dashed, rounded corners=4pt, xshift=-2.6cm] (-0.7,0.03) -- (-0.4,0.23)--(0.4,0.23) --(0.7,0.03);
 \draw[thick, green, rounded corners=4pt, xshift=0cm] (-0.7,-0.03) -- (-0.4,-0.23)--(0.4,-0.23);
 \draw[thick, green, dashed, rounded corners=4pt, xshift=0cm] (-0.7,0.03) -- (-0.4,0.23)--(0.4,0.23);
 \draw[thick, green, rounded corners=4pt, xshift=5.2cm] (-0.4,-0.23)--(0.4,-0.23) --(0.7,-0.03);
 \draw[thick, green, dashed, rounded corners=4pt, xshift=5.2cm]  (-0.4,0.23)--(0.4,0.23) --(0.7,0.03);
 \draw[thick, green, rounded corners=4pt, xshift=7.8cm] (-0.7,-0.03) -- (-0.4,-0.23)--(0.4,-0.23) --(0.7,-0.03);
 \draw[thick, green, dashed, rounded corners=4pt, xshift=7.8cm] (-0.7,0.03) -- (-0.4,0.23)--(0.4,0.23) --(0.7,0.03);

\draw[thick, blue, rounded corners=4pt, xshift=-7.8cm] (-0.7,-0.03) -- (-0.4,-0.23)--(0.4,-0.23) --(0.7,-0.03);
 \draw[thick, blue, dashed, rounded corners=4pt, xshift=-7.8cm] (-0.7,0.03) -- (-0.4,0.23)--(0.4,0.23) --(0.7,0.03);
\draw[thick, blue, rounded corners=4pt] (-3.93,0.6) -- (-4.13, 1)--(-4.13,1.6) --(-3.93,2);
 \draw[thick, blue, dashed, rounded corners=4pt] (-3.87,0.6) -- (-3.67, 1)--(-3.67,1.6) --(-3.87,2);
\draw[thick, blue, rounded corners=4pt, xshift=2.63cm] (-3.93,0.6) -- (-4.13, 1)--(-4.13,1.6) --(-3.93,2);
 \draw[thick, blue, dashed, rounded corners=4pt, xshift=2.57cm] (-3.87,0.6) -- (-3.67, 1)--(-3.67,1.6) --(-3.87,2);
\draw[thick, blue, rounded corners=4pt, xshift=10.4cm] (-0.7,-0.03) -- (-0.4,-0.23)--(0.4,-0.23) --(0.7,-0.03);
 \draw[thick, blue, dashed, rounded corners=4pt, xshift=10.4cm] (-0.7,0.03) -- (-0.4,0.23)--(0.4,0.23) --(0.7,0.03);

\node[blue, scale=0.9] at (-7.6 ,-0.7) {$a_1$};
\node[blue, scale=0.9] at (-4.7,1.3) {$a_2$};
\node[blue, scale=0.9] at (-2.1,1.3) {$a_3$};
\node[green, scale=0.9] at (-5.2,-0.7) {$c_1$};
\node[green, scale=0.9] at (-2.6,-0.7) {$c_2$};
\node[green, scale=0.9] at (0,-0.7) {$c_3$};
\node[green, scale=0.9] at (7.8,-1.1) {$c_{g-1}$};
\node[blue, scale=0.9] at (10.3,-0.7) {$a_g$};

\node[red, scale=0.9] at (-6.5,1.35) {$b_1$};
\node[red, scale=0.9] at (-3,1) {$b_2$};
\node[red, scale=0.9] at (-0.4,1) {$b_3$};
\node[red, scale=0.9] at (8.2,1.1) {$b_g$};
\node[red, scale=0.9] at (5.4,1.1) {$b_{g-1}$};
\node[red, scale=0.9] at (4,-1.35) {$b_{g+1}$};

\draw[very thick, violet, ->] (1.3,2.1)--(1.3,3.5);
 \node[scale=1, violet] at (1.8,3.2) {$z$};
\draw[very thick, violet, ->] (11.45,0)--(12.6,0);
 \node[scale=1, violet] at (12.3,-0.5) {$y$};
\end{scope}
\node[scale=1] at (-3.7,-4) {(i)};
\node[scale=1] at (4,-4) {(ii)};
\node[scale=1] at (0,-8.3) {(iii)};

\end{tikzpicture}
\caption{The three models. The $x$--axis is perpendicular to the page in all.}
\label{fig:models}
\end{figure}
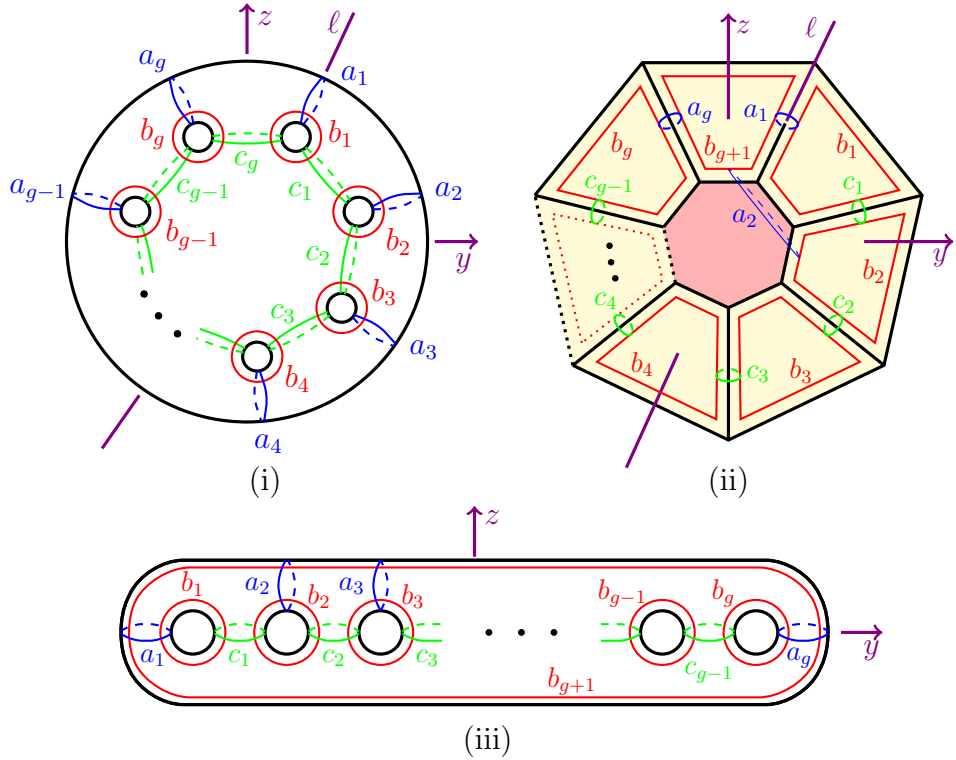

Consider the three different embeddings of the closed surface $\Sigma_g$ in $\R^3$ as depicted in Figure~\ref{fig:models}. The surface in (ii) is the boundary of the solid handlebody which consists of two thickened $(g+1)$--gons, stacked on top of each other, and $(g+1)$ solid handles joining their corresponding vertices. There are orientation-preserving diffeomorphisms between the three models which identify the curves labeled as $a_i, b_j, c_k$ in each one. These models allow us to easily introduce and study certain torsion elements in $\MCG$ coming from the symmetries of the surface. The surface $\Sigma_g$ is invariant under the following maps:

\begin{itemize}
\item  the clockwise $\frac{2\pi}{g}$--rotation $R$ about the $x$--axis in Figure~\ref{fig:models}(i),
\item the clockwise $\frac{2\pi}{g+1}$--rotation $S$ about the $x$--axis in Figure~\ref{fig:models}(ii),
\item  the rotations $\rho_1$ and $\rho_2$ by $\pi$ about the $z$--axis and the line $\ell$, respectively,
 in Figure~\ref{fig:models}(i),  
\item the rotations $\sigma_1$ and $\sigma_2$ by $\pi$ about the $z$--axis and the line $\ell$, respectively,
 in Figure~\ref{fig:models}(ii), 
\item the rotations $\sigma$ and $h$ by $\pi$  about the $z$--axis and the $y$--axis, respectively,
in Figure~\ref{fig:models}(iii).
\end{itemize}

Clearly, in $\MCG$, the rotations $R$ and $S$ yield torsion elements of orders $g$ and $g+1$, respectively, and
 $\rho_i, \sigma_i, \sigma, h$  yield involutions (elements of order  $2$), where $h$ is a \emph{hyperelliptic involution}. It is easy to check (say by Alexander's method applied to the maximal chain $a_1, b_1, c_1, c_2, \ldots, c_{g-1}, b_g$) that  the involutions $\sigma$ and $h$ we described on the model~(iii) correspond on the model~(ii) to $\sigma_1$ and the involution $h_1$ which interchanges the top and bottom thickened $(g+1)$--gons by a rotation along a central circle through the mid-points of the solid handles. Under these identifications, we define one more  torsion element:
\begin{itemize}
\item $T$ is the ``alternating rotation'' of the surface in Figure~\ref{fig:models}(ii), prescribed as $T=S h_1$.  
\end{itemize}
Here $S$ and $h_1$ commute, so $T= h_1 S$ as well. Note that $T$ is of order $g+1$ if $g$ is odd, and  $2(g+1)$ if $g$ is even.

\begin{proposition} \label{torsioncomm}
In $\MCG$, the mapping classes $R$ when $g= 2k +1$, $k \geq 1$, and $S, h, T$ when $g= 2k$, $k \geq 1$, are all commutators, which can be expressed as
\begin{enumerate}
\item $R = [R^{k+1}, \rho_1]$,
\item $S = [S^{k+1} , \sigma_1]$,
\item \,$h = \,[\sigma, \, P^{-(2k+1)}]$,
\item $T =[ S^{k+1} P^{2k+1} \, , P^{-(2k+1)} \sigma_1 P^{2k+1}]$, 
\end{enumerate}
where $P=A_1 B_1 (C_1 B_2) \cdots (C_{k-1} B_{k}) $. 
\end{proposition}

\begin{proof}
From the dihedral symmetries of the models (i) and (ii) in Figure~\ref{fig:models}, we easily deduce that  

\indent $R = \rho_2 \rho_1$ and $\rho_2 = R^{k+1} \rho_1 \, R^{-{(k+1)}}$, when $g= 2k+1 \geq 3$, and

\indent $S = \sigma_2 \sigma_1$ and $\sigma_2 = S^{k+1} \sigma_1 S^{-(k+1)}$, when $g= 2k \geq 2$. 

\noindent Therefore $R= \rho_2 \rho_1 = (R^{k+1} \rho_1 \, R^{-{(k+1)}}) \rho_1 = [R^{k+1}, \rho_1]$, as $\rho_1^{-1} = \rho_1$, and similarly $S =  [S^{k+1} , \sigma_1]$. This proves (1) and (2).

For (3), let $g=2k$ and let $\delta$ be the separating curve that is the intersection of $\Sigma_g$ with the $xz$--plane in Figure~\ref{fig:models}(iii), such that $\delta$ is the common boundary of two compact genus--$k$ subsurfaces $\Sigma$ and $\Sigma'$. The surfaces $\Sigma$ and $\Sigma'$ are tubular neighborhoods of the $(2k)$--chains $a_1, b_1, c_1, b_2, \ldots, c_{k-1}, b_k$ and $a_g, b_g, c_{g-1}, b_{g-1}, \ldots c_{k+1}, b_{k+1}$.

We first show that the hyperelliptic involution $h$ can be expressed as
\[ 
h =(A_1 B_1 C_1 B_2 \cdots C_{k-1} B_k )^{2k+1} (A_g B_g C_{g-1} B_{g-1}\cdots C_{k+1} B_{k+1})^{-(2k+1)}  \, . 
\]
Let $P=A_1 B_1 C_1 B_2 \cdots C_{k-1} B_k$ and $P'= A_g  B_g C_{g-1} B_{g-1} \cdots C_{k+1} B_{k+1}$. Note that the diffeomorphism $P^{2k+1}$ is a $\pi$--rotation of the subsurface $\Sigma$ along the $y$--axis, followed by isotoping the boundary back to its original position, so that its square $P^{4k+2}= t_\delta$. (The latter equality is known as the \emph{$(2k)$--chain relation}; see e.g. \cite{FarbMargalit}.) This can be easily checked by the Alexander's method:  $P^{2k+1}$ maps each one of the curves $a_1,b_1,c_1,\ldots,c_{k-1},b_k$ to itself, but with reversed orientation, whereas it maps the arc $\alpha$ to the arc $\beta$ in Figure~\ref{fig:hypellinv}. By the same token, $(P')^{2k+1}$ is a  similar $\pi$--rotation of the subsurface $\Sigma'$ along the $y$--axis, albeit in the opposite direction. So taking the inverse of one, as we did above, we get a \mbox{$\pi$--rotation} of the whole surface $\Sigma_g = \Sigma \cup_{\delta} \Sigma'$, which is the hyperelliptic involution $h$.

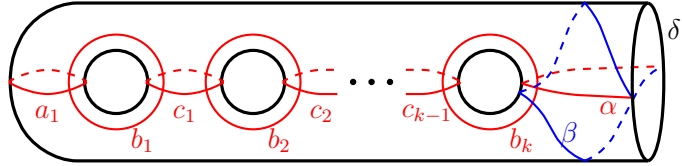
\begin{figure}[h!]	
\centering
\begin{tikzpicture}[every node/.style={inner sep=0pt}, scale=0.7]
 \draw[very thick,rounded corners=15pt] (3.6,1.5) -- (-8,1.5)--(-8.7,0)-- (-8,-1.5) -- (3.6,-1.5);
 \draw[very thick, xshift=-6.5cm] (0,0) circle [radius=0.6cm]; 
 \draw[very thick, xshift=-3.9cm] (0,0) circle [radius=0.6cm];
 \draw[very thick,fill, xshift=-2cm] (0,0) circle [radius=0.03cm];
 \draw[very thick,fill, xshift=-1.65cm] (0,0) circle [radius=0.03cm];
 \draw[very thick,fill, xshift=-1.3cm] (0,0) circle [radius=0.03cm];
 \draw[very thick, xshift=0.6cm] (0,0) circle [radius=0.6cm];
 \draw[thick, red, xshift=-6.5cm] (0,0) circle [radius=0.9cm];
\draw[thick, red, xshift=-3.9cm] (0,0) circle [radius=0.9cm];
\draw[thick, red, xshift=0.6cm] (0,0) circle [radius=0.9cm];
\draw[thick, red,  rounded corners=8pt, xshift=-7.8cm] (-0.7,-0.03) -- (-0.4,-0.23)--(0.4,-0.23) --(0.7,-0.03);
 \draw[thick, red,  dashed, rounded corners=8pt, xshift=-7.8cm] (-0.7,0.03) -- (-0.4,0.23)--(0.4,0.23) --(0.7,0.03);
 \draw[thick, red,  rounded corners=8pt, xshift=-5.2cm] (-0.7,-0.03) -- (-0.4,-0.23)--(0.4,-0.23) --(0.7,-0.03);
 \draw[thick, red,  dashed, rounded corners=8pt, xshift=-5.2cm] (-0.7,0.03) -- (-0.4,0.23)--(0.4,0.23) --(0.7,0.03);
 \draw[thick, red,  rounded corners=8pt, xshift=-2.6cm] (-0.7,-0.03) -- (-0.4,-0.23)--(0.3,-0.23) ;
 \draw[thick, red,  dashed, rounded corners=8pt, xshift=-2.6cm] (-0.7,0.03) -- (-0.4,0.23)--(0.3,0.23);
\draw[thick, red,  rounded corners=8pt, xshift=-0.7cm]  (-0.3,-0.23)--(0.4,-0.23) --(0.7,-0.03);
 \draw[thick, red,  dashed, rounded corners=8pt, xshift=-0.7cm]  (-0.3,0.23)--(0.4,0.23) --(0.7,0.03);
  \draw[thick, red,  rounded corners=8pt, xshift=-2.6cm] (3.8,-0.03) -- (4.5,-0.23)--(5.9,-0.3) ;
  \draw[thick,  red, dashed, rounded corners=8pt, xshift=-2.6cm] (3.8,-0.03) -- (4.5,0.23)--(6.5,0.3) ;
\draw[thick, blue, rounded corners=6pt, xshift=-2.6cm] (3.75,-0.2) -- (4.1,-0.43)--(4.6,-1.3) --(5,-1.5);
\draw[thick, blue, dashed, rounded corners=6pt, xshift=-2.6cm] (5,-1.5) -- (5.4,-1.3) --(6.2,0.2)--(6.5,0.3) ;
\draw[thick, blue, dashed, rounded corners=6pt, xshift=-2.6cm] (3.75,-0.2) -- (4.3,0) --(4.8,1.3)--(5,1.5) ;
\draw[thick, blue,  rounded corners=6pt, xshift=-2.6cm] (5,1.5) -- (5.2,1.3) --(5.7,0)--(5.9,-0.3) ;
\draw[very thick,  rounded corners=4pt] (3.6,1.5) ..controls (3.2,1.4) and (3.2,-1.4)..(3.6,-1.5);
\draw[very thick,  rounded corners=4pt] (3.6,1.5) ..controls (4,1.4) and (4,-1.4)..(3.6,-1.5);
\node[red, scale=0.9] at (-7.8 ,-0.6) {$a_1$};
\node[red, scale=0.9] at (-5.2 ,-0.6) {$c_1$};
\node[red, scale=0.9] at (-2.6 ,-0.6) {$c_2$};
\node[red, scale=0.9] at (-0.6 ,-0.6) {$c_{k-1}$};
\node[red, scale=0.9] at (-6 ,-1.1) {$b_{1}$};
\node[red, scale=0.9] at (-3.4 ,-1.1) {$b_2$};
\node[red, scale=0.9] at (1.2 ,-1.1) {$b_k$};
\node[red, scale=0.9] at (2.85 ,-0.5) {$\alpha$};
\node[blue, scale=0.9] at (2.1 ,-1) {$\beta$};
\node[scale=0.9]  at (4.1 ,1) {$\delta$};
\end{tikzpicture}
\caption{The genus--$k$ subsurface $\Sigma$ with boundary $\delta$.}
\label{fig:hypellinv}
\end{figure}

Next, we observe that the involution $\sigma$ on $\Sigma_g$ interchanges these two $2k$--chains. It follows that $P'= \sigma P \sigma^{-1}$, and therefore 
\begin{equation*}
h=(\sigma P \sigma^{-1})^{-(2k+1)} \, P^{2k+1}  = \sigma P^{-(2k+1)} \sigma^{-1} P^{2k+1} = [ \sigma , P^{-(2k+1)}  ]. 
\end{equation*}

Lastly, since $\sigma$ and $h$ correspond to $\sigma_1$ and $h_1$ in model (ii), we have
\begin{eqnarray*}
T & = & S h_1 \\
&=& [S^{k+1} , \sigma_1] [\sigma_1, \, P^{-(2k+1)}]  \\
&=& (S^{k+1}  \sigma_1 S^{-(k+1)} \sigma_1^{-1}) (\sigma_1  P^{-(2k+1)} \sigma_1^{-1} P^{2k+1}) \\
&=& S^{k+1}  \sigma_1 S^{-(k+1)}   P^{-(2k+1)} \sigma_1^{-1} P^{2k+1} \\
&=& S^{k+1} (P^{2k+1} P^{-(2k+1)}) \sigma_1 (P^{2k+1} P^{-(2k+1)}) S^{-(k+1)} P^{-(2k+1)} \sigma_1^{-1} P^{2k+1} \\
&=& (S^{k+1} P^{2k+1}) (P^{-(2k+1)} \sigma_1 P^{2k+1}) (S^{k+1} P^{2k+1} )^{-1} (P^{-(2k+1)} \sigma_1 P^{2k+1})^{-1} \\
&=& [ S^{k+1} P^{2k+1} \, , P^{-(2k+1)} \sigma_1 P^{2k+1}],
\end{eqnarray*} 
which\footnote{In general, $[x, y] [y, z] = [ x z^{-1}, zyz^{-1}]$ for \emph{any} group elements $x, y, z \in G$ \cite{Carmichael}.}
concludes (4).
\end{proof}


\bigskip
\section{New generating sets for the mapping class group}\label{sec:gen}

Here we obtain several new generating sets for $\MCG$, focusing on generators that can be expressed as commutators (for suitable $g$).

\begin{lemma} \label{lem:G=Mod} 
For $g\geq 3$, the mapping class group  
$\mod(\Sigma_g)$ is generated by $A_1B_1^{-1}, A_2B_2^{-1}, B_1C_1^{-1},$
$B_iB_{i+1}^{-1}$ and  $C_jC_{j+1}^{-1}$, where $1\leq i\leq g-1$ and $1\leq j\leq g-2$. 
\end{lemma}
\begin{proof}
Let $\Gamma$ be the subgroup of $\mod(\Sigma_g)$ generated by the set
\[
\{ A_1B_1^{-1}, A_2B_2^{-1}, B_1C_1^{-1}, B_iB_{i+1}^{-1},C_jC_{j+1}^{-1} \, \}_{\forall i, j}.
\]
Then
\begin{equation*}
\{ A_2 A_1^{-1}, B_iB_j^{-1}, C_iC_j^{-1}, B_iC_j^{-1}, A_1B_i^{-1}, A_2B_i^{-1}, A_1C_j^{-1}, A_2C_j^{-1}\}_{\forall i, j} \subset \Gamma.
\end{equation*} 
\noindent For example, we have $A_2 C_1^{-1} =(A_2 B_2^{-1}) (B_1 B_2^{-1})^{-1}(B_1 C_1^{-1}) \in \Gamma$. Since $C_j C_{j+1}^{-1} \in \Gamma$, multiplying these elements in increasing index, we get $C_1 C_j^{-1} \in \Gamma$. So $A_2 C_j^{-1} \in \Gamma$. The others can be easily verified in a similar fashion.

\begin{figure}[hbt]	
\centering
\begin{tikzpicture}[every node/.style={inner sep=0pt}, scale=0.85]
\begin{scope} [yshift=0cm, scale=0.6]
 \draw[very thick,rounded corners=14pt] (-7.8,-2) -- (7.8,-2)--(8.75,0)-- (7.8,2) -- (-7.8,2) --(-8.75, 0)-- cycle;
 \draw[very thick, xshift=1.3cm] (0,0) circle [radius=0.6cm];
 \draw[very thick,fill, xshift=3.4cm] (0,0) circle [radius=0.02cm];
 \draw[very thick,fill, xshift=3.9cm] (0,0) circle [radius=0.02cm];
 \draw[very thick,fill, xshift=4.4cm] (0,0) circle [radius=0.02cm];
 
 \draw[very thick, xshift=6.5cm] (0,0) circle [radius=0.6cm];
 \draw[very thick, xshift=-1.3cm] (0,0) circle [radius=0.6cm];
 \draw[very thick, xshift=-3.9cm] (0,0) circle [radius=0.6cm];
 \draw[very thick, xshift=-6.5cm] (0,0) circle [radius=0.6cm];

\draw[thick, blue, rounded corners=4pt, xshift=-7.8cm] (-0.7,-0.03) -- (-0.4,-0.23)--(0.4,-0.23) --(0.7,-0.03);
 \draw[thick, blue, dashed, rounded corners=4pt, xshift=-7.8cm] (-0.7,0.03) -- (-0.4,0.23)--(0.4,0.23) --(0.7,0.03);
 \draw[thick, blue, rounded corners=4pt, xshift=-5.2cm] (-0.7,-0.03) -- (-0.4,-0.23)--(0.4,-0.23) --(0.7,-0.03);
 \draw[thick, blue, dashed, rounded corners=4pt, xshift=-5.3cm] (-0.7,0.03) -- (-0.4,0.23)--(0.4,0.23) --(0.7,0.03);
 \draw[thick, blue, rounded corners=4pt, xshift=-2.6cm] (-0.7,-0.03) -- (-0.4,-0.23)--(0.4,-0.23) --(0.7,-0.03);
 \draw[thick, blue, dashed, rounded corners=4pt, xshift=-2.6cm] (-0.7,0.03) -- (-0.4,0.23)--(0.4,0.23) --(0.7,0.03);
\draw[thick, red, rounded corners=3pt] (-3.93,0.6) -- (-4.13, 1)--(-4.13,1.6) --(-3.93,2);
 \draw[thick, red, dashed, rounded corners=3pt] (-3.87,0.6) -- (-3.67, 1)--(-3.67,1.6) --(-3.87,2);
\draw[thick, blue, rounded corners=3pt] (-1.33,0.6) -- (-1.53, 1)--(-1.53,1.6) --(-1.33,2);
 \draw[thick, blue, dashed, rounded corners=3pt] (-1.27,0.6) -- (-1.07, 1)--(-1.07,1.6) --(-1.27,2);
\node[blue] at (-7.6 ,-0.7) {$a_1$};
\node[blue] at (-5 ,-0.7) {$c_1$};
\node[blue] at (-2.4 ,-0.7) {$c_2$};
\node[red] at (-4.7,1.2) {$a_2$};
\node[blue] at (-2.1,1.2) {$a_3$};
\end{scope}

\begin{scope} [yshift=-2.8cm, scale=0.6]
 \draw[very thick,rounded corners=14pt] (-7.8,-2) -- (7.8,-2)--(8.75,0)-- (7.8,2) -- (-7.8,2) --(-8.75, 0)-- cycle;
 \draw[very thick, xshift=1.3cm] (0,0) circle [radius=0.6cm];
 \draw[very thick,fill, xshift=3.4cm] (0,0) circle [radius=0.02cm];
 \draw[very thick,fill, xshift=3.9cm] (0,0) circle [radius=0.02cm];
 \draw[very thick,fill, xshift=4.4cm] (0,0) circle [radius=0.02cm];
 
 \draw[very thick, xshift=6.5cm] (0,0) circle [radius=0.6cm];
 \draw[very thick, xshift=-1.3cm] (0,0) circle [radius=0.6cm];
 \draw[very thick, xshift=-3.9cm] (0,0) circle [radius=0.6cm];
 \draw[very thick, xshift=-6.5cm] (0,0) circle [radius=0.6cm];
\draw[thick, red, rounded corners=8pt]    (-6,0.3) -- (-5.5, 1.2)--(-3.7, 1.5)--(-1.9,1.2) --(-1.5,0.53);
\draw[thick, red,  dashed,  rounded corners=3pt] (-5.9,0.25) --(-5.7,0.3)-- (-5, 1.1)--(-3.9,1.2)--(-2.8,1.1)--(-1.9,0.6) --(-1.6,0.5);
\draw[thick, red, dashed, rounded corners=4pt, xshift=-2.6cm] (-0.7,-0.03) -- (-0.4,-0.23)--(0.4,-0.23) --(0.7,-0.03);
\draw[thick, red, dashed, rounded corners=3pt] (-6.53,0.6) -- (-6.73, 1)--(-6.73,1.6) --(-6.53,2);
\draw[thick, red, rounded corners=4pt] (-6.4,2) --(-6,1.5) --(-3.5, 1.2)--(-2.2,0.4)--(-1.9,0.05);
\draw[thick, red, rounded corners=6pt] (-6.4,0.6) --(-6,1) --(-4.3, 1.1)--(-2.8, 0.6)--(-3.3,0.05);

\node[red] at (-1.5 ,1.5) {$d_1$};
\node[red] at (-7.3,1.2) {$d_2$};
\end{scope}
\end{tikzpicture}
\caption{The curves of the embedded lantern relation $A_1C_1C_2A_3=A_2D_1D_2$ in $\Sigma_g$.}  
\label{fig:lanternembed}
\end{figure}
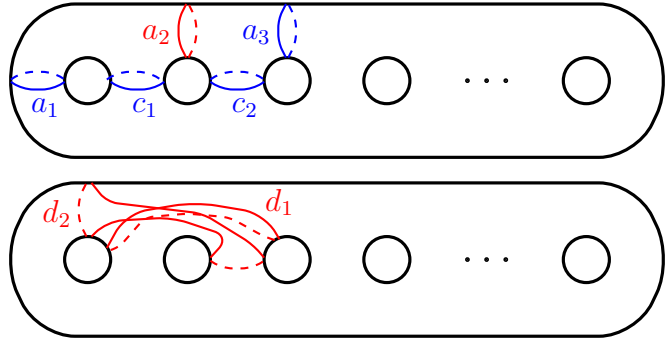

By the lantern relation, the following holds in $\MCG$:
\[
A_1C_1C_2A_3=A_2D_1D_2,
\]
where the curves are as in Figure~\ref{fig:lanternembed}. We can rewrite this relation as
\begin{equation}
   A_3=(A_2 C_1^{-1})(D_1 C_2^{-1})(D_2 A_1^{-1}), \label{eqn:lantern}
\end{equation}
as $C_1, C_2$ and $A_3$ commute with all the other Dehn twists here.
Note that $A_2 C_1^{-1}\in \Gamma$. One can check that the diffeomorphism 
\[
F=(A_1B_2^{-1}) (A_1C_1^{-1}) (A_1C_2^{-1}) (A_1B_2^{-1})
\] 
maps the pair of simple closed curves $(a_2,a_1)$ to $(d_2,a_1)$. Since $F \in \Gamma$, we have $F(A_2A_1^{-1})F^{-1}=D_2A_1^{-1} \in \Gamma$ by the conjugation relation.  We also have
$D_2C_2^{-1}=(D_2A_1^{-1})(A_1C_2^{-1})\in \Gamma$.

Likewise, the diffeomorphism
\[
G=(C_2B_1^{-1}) (C_2A_1^{-1}) (C_2C_1^{-1}) (C_2B_1^{-1})
\]
maps the pair  of simple closed curves $(d_2,c_2)$ to $(d_1,c_2)$, and is in $\Gamma$. Therefore, the element
$G(D_2C_2^{-1})G^{-1}=D_1C_2^{-1}$ is also in $\Gamma$, once again by the conjugation relation.

Now the equality~\eqref{eqn:lantern} implies that the Dehn twist $A_3$ is in $\Gamma$.
Also, $A_3 (B_3B_1^{-1}) A_3  (B_1B_3^{-1})A_3 ^{-1}=B_3$ is in  $\Gamma$. 
It now follows easily
that $A_1, A_2, B_1,B_2,\ldots, B_g$ and $C_1,C_2,\dots, C_{g-1}$ are all contained in $\Gamma$. 
We conclude from Theorem~\ref{thm:DLK} that $\Gamma=\mod(\Sigma_g)$.
 \end{proof}

We now present various new generators for the mapping class group $\MCG$ we need for our main theorem.

\begin{theorem} \label{prop:g=odd} 
For $g\geq 5$, the mapping class group  $\mod(\Sigma_g)$ is generated by 
$R$ and $A_1A_2C_2^{-1}B_4^{-1}.$
\end{theorem}

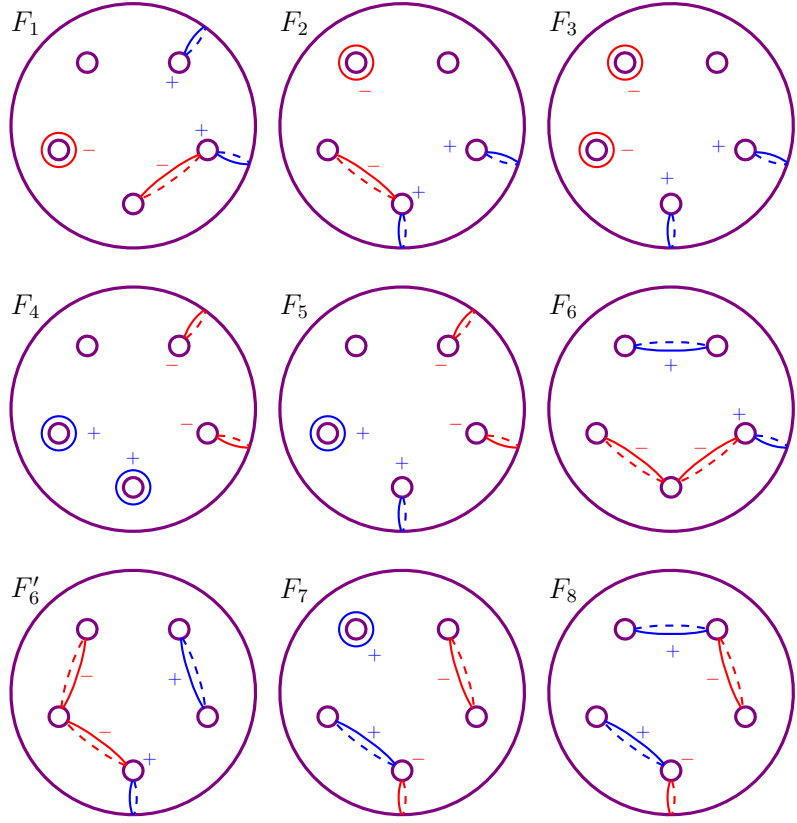
\begin{figure}
\begin{tikzpicture}[scale=0.65]
\begin{scope} [xshift=0cm, yshift=0cm]
 \draw[very thick, violet] (0,0) circle [radius=2.5cm];
 \draw[very thick, violet, rotate=36] (0,1.6) circle [radius=0.2cm]; 
 \draw[very thick, violet, rotate=108] (0,1.6) circle [radius=0.2cm];  
 \draw[very thick, violet, rotate=-36] (0,1.6) circle [radius=0.2cm]; 
 \draw[very thick, violet, rotate=-108] (0,1.6) circle [radius=0.2cm]; 
 \draw[very thick, violet, rotate=-180] (0,1.6) circle [radius=0.2cm];  
 \draw[thick, red, rotate=-144, rounded corners=4pt] (-0.76,1.27)--(-0.4, 1.2)--(0.4,1.2)--(0.76,1.27);  
 \draw[thick, red, dashed,  rotate=-144, rounded corners=4pt] (-0.76,1.3)--(-0.4, 1.37)--(0.4,1.37)--(0.76,1.3);  
 \draw[thick, red, rotate=108] (0,1.6) circle [radius=0.35cm]; 
 \draw[thick, blue, rotate=-36, rounded corners=4pt] (-0.02,1.8)--(-0.1, 2.15)--(-0.02,2.5);  
 \draw[thick, blue, dashed, rotate=-36, rounded corners=4pt]  (0.02,1.8)--(0.1, 2.15)--(0.02,2.5); 
 \draw[thick, dashed, blue, rotate=-108, rounded corners=4pt] (-0.02,1.8)--(-0.1, 2.15)--(-0.02,2.5);  
 \draw[thick, blue, rotate=-108, rounded corners=4pt]  (0.02,1.8)--(0.1, 2.15)--(0.02,2.5); 
  \node[scale=0.6, blue] at (0.8,0.9) {$+$};
  \node[scale=0.6, red] at (-0.9,-0.5) {$-$};
  \node[scale=0.6, red] at (0.6,-0.8) {$-$};
  \node[scale=0.6, blue] at (1.4,-0.1) {$+$};
\end{scope}

\begin{scope} [xshift=5.5cm, yshift=0cm, rotate=-72]
  \draw[very thick, violet] (0,0) circle [radius=2.5cm];
 \draw[very thick, violet, rotate=36] (0,1.6) circle [radius=0.2cm]; 
 \draw[very thick, violet, rotate=108] (0,1.6) circle [radius=0.2cm];  
 \draw[very thick, violet, rotate=-36] (0,1.6) circle [radius=0.2cm]; 
 \draw[very thick, violet, rotate=-108] (0,1.6) circle [radius=0.2cm]; 
 \draw[very thick, violet, rotate=-180] (0,1.6) circle [radius=0.2cm];  
 \draw[thick, red, rotate=-144, rounded corners=4pt] (-0.76,1.27)--(-0.4, 1.2)--(0.4,1.2)--(0.76,1.27);  
 \draw[thick, red, dashed,  rotate=-144, rounded corners=4pt] (-0.76,1.3)--(-0.4, 1.37)--(0.4,1.37)--(0.76,1.3);  
 \draw[thick, red, rotate=108] (0,1.6) circle [radius=0.35cm]; 
 \draw[thick, blue, rotate=-36, rounded corners=4pt] (-0.02,1.8)--(-0.1, 2.15)--(-0.02,2.5);  
 \draw[thick, blue, dashed, rotate=-36, rounded corners=4pt]  (0.02,1.8)--(0.1, 2.15)--(0.02,2.5); 
 \draw[thick, dashed, blue, rotate=-108, rounded corners=4pt] (-0.02,1.8)--(-0.1, 2.15)--(-0.02,2.5);  
 \draw[thick, blue, rotate=-108, rounded corners=4pt]  (0.02,1.8)--(0.1, 2.15)--(0.02,2.5); 
  \node[scale=0.6, blue] at (0.7,0.8) {$+$};
  \node[scale=0.6, red] at (-0.9,-0.5) {$-$};
  \node[scale=0.6, red] at (0.6,-0.8) {$-$};
  \node[scale=0.6, blue] at (1.4,-0.1) {$+$};
   \end{scope}

\begin{scope} [xshift=11cm, yshift=0cm, rotate=-72]
\draw[very thick, violet] (0,0) circle [radius=2.5cm];
 \draw[very thick, violet, rotate=36] (0,1.6) circle [radius=0.2cm]; 
 \draw[very thick, violet, rotate=108] (0,1.6) circle [radius=0.2cm];  
 \draw[very thick, violet, rotate=-36] (0,1.6) circle [radius=0.2cm]; 
 \draw[very thick, violet, rotate=-108] (0,1.6) circle [radius=0.2cm]; 
 \draw[very thick, violet, rotate=-180] (0,1.6) circle [radius=0.2cm];  
 \draw[thick, red, rotate=108] (0,1.6) circle [radius=0.35cm]; 
\draw[thick, red, rotate=180] (0,1.6) circle [radius=0.35cm]; 
 \draw[thick, blue, rotate=-36, rounded corners=4pt] (-0.02,1.8)--(-0.1, 2.15)--(-0.02,2.5);  
 \draw[thick, blue, dashed, rotate=-36, rounded corners=4pt]  (0.02,1.8)--(0.1, 2.15)--(0.02,2.5); 
 \draw[thick, dashed, blue, rotate=-108, rounded corners=4pt] (-0.02,1.8)--(-0.1, 2.15)--(-0.02,2.5);  
 \draw[thick, blue, rotate=-108, rounded corners=4pt]  (0.02,1.8)--(0.1, 2.15)--(0.02,2.5); 
  \node[scale=0.6, blue] at (0.7,0.8) {$+$};
  \node[scale=0.6, red] at (-0.9,-0.5) {$-$};
  \node[scale=0.6, red] at (0.2,-1) {$-$};
  \node[scale=0.6, blue] at (1,-0.4) {$+$};
    \end{scope}

\begin{scope} [xshift=0cm, yshift=-5.8cm]
 \draw[very thick, violet] (0,0) circle [radius=2.5cm];
 \draw[very thick, violet, rotate=36] (0,1.6) circle [radius=0.2cm]; 
 \draw[very thick, violet, rotate=108] (0,1.6) circle [radius=0.2cm];  
 \draw[very thick, violet, rotate=-36] (0,1.6) circle [radius=0.2cm]; 
 \draw[very thick, violet, rotate=-108] (0,1.6) circle [radius=0.2cm]; 
 \draw[very thick, violet, rotate=-180] (0,1.6) circle [radius=0.2cm];  
 \draw[thick, blue, rotate=108] (0,1.6) circle [radius=0.35cm]; 
\draw[thick, blue, rotate=180] (0,1.6) circle [radius=0.35cm]; 
 \draw[thick, red, rotate=-36, rounded corners=4pt] (-0.02,1.8)--(-0.1, 2.15)--(-0.02,2.5);  
 \draw[thick, red, dashed, rotate=-36, rounded corners=4pt]  (0.02,1.8)--(0.1, 2.15)--(0.02,2.5); 
 \draw[thick, dashed, red, rotate=-108, rounded corners=4pt] (-0.02,1.8)--(-0.1, 2.15)--(-0.02,2.5);  
 \draw[thick, red, rotate=-108, rounded corners=4pt]  (0.02,1.8)--(0.1, 2.15)--(0.02,2.5); 
  \node[scale=0.6, red] at (0.8,0.9) {$-$};
  \node[scale=0.6, blue] at (-0.8,-0.5) {$+$};
  \node[scale=0.6, blue] at (0,-1) {$+$};
  \node[scale=0.6, red] at (1.1,-0.3) {$-$};
\end{scope}

\begin{scope} [xshift=5.5cm, yshift=-5.8cm, rotate=0]
 \draw[very thick, violet] (0,0) circle [radius=2.5cm];
 \draw[very thick, violet, rotate=36] (0,1.6) circle [radius=0.2cm]; 
 \draw[very thick, violet, rotate=108] (0,1.6) circle [radius=0.2cm];  
 \draw[very thick, violet, rotate=-36] (0,1.6) circle [radius=0.2cm]; 
 \draw[very thick, violet, rotate=-108] (0,1.6) circle [radius=0.2cm]; 
 \draw[very thick, violet, rotate=-180] (0,1.6) circle [radius=0.2cm];  
 \draw[thick, blue, rotate=108] (0,1.6) circle [radius=0.35cm]; 
 \draw[thick, red, rotate=-36, rounded corners=4pt] (-0.02,1.8)--(-0.1, 2.15)--(-0.02,2.5);  
 \draw[thick, red, dashed, rotate=-36, rounded corners=4pt]  (0.02,1.8)--(0.1, 2.15)--(0.02,2.5); 
 \draw[thick, dashed, red, rotate=-108, rounded corners=4pt] (-0.02,1.8)--(-0.1, 2.15)--(-0.02,2.5);  
 \draw[thick, red, rotate=-108, rounded corners=4pt]  (0.02,1.8)--(0.1, 2.15)--(0.02,2.5); 
 \draw[thick, dashed, blue, rotate=180, rounded corners=4pt] (-0.02,1.8)--(-0.1, 2.15)--(-0.02,2.5);  
 \draw[thick, blue, rotate=180, rounded corners=4pt]  (0.02,1.8)--(0.1, 2.15)--(0.02,2.5); 
  \node[scale=0.6, red] at (0.8,0.9) {$-$};
  \node[scale=0.6, blue] at (-0.8,-0.5) {$+$};
  \node[scale=0.6, blue] at (0,-1.1) {$+$};
  \node[scale=0.6, red] at (1.1,-0.3) {$-$};
 \end{scope}

\begin{scope} [xshift=11cm, yshift=-5.8cm, rotate=0]
 \draw[very thick, violet] (0,0) circle [radius=2.5cm];
 \draw[very thick, violet, rotate=36] (0,1.6) circle [radius=0.2cm]; 
 \draw[very thick, violet, rotate=108] (0,1.6) circle [radius=0.2cm];  
 \draw[very thick, violet, rotate=-36] (0,1.6) circle [radius=0.2cm]; 
 \draw[very thick, violet, rotate=-108] (0,1.6) circle [radius=0.2cm]; 
 \draw[very thick, violet, rotate=-180] (0,1.6) circle [radius=0.2cm];  
 \draw[thick, red, rotate=-144, rounded corners=4pt] (-0.76,1.27)--(-0.4, 1.2)--(0.4,1.2)--(0.76,1.27);  
 \draw[thick, red, dashed,  rotate=-144, rounded corners=4pt] (-0.76,1.3)--(-0.4, 1.37)--(0.4,1.37)--(0.76,1.3);  
 \draw[thick, blue, rotate=0, rounded corners=4pt] (-0.76,1.27)--(-0.4, 1.2)--(0.4,1.2)--(0.76,1.27);  
 \draw[thick, blue, dashed,  rotate=0, rounded corners=4pt] (-0.76,1.3)--(-0.4, 1.37)--(0.4,1.37)--(0.76,1.3);  
 \draw[thick, red, rotate=144, rounded corners=4pt] (-0.76,1.27)--(-0.4, 1.2)--(0.4,1.2)--(0.76,1.27);  
 \draw[thick, red, dashed,  rotate=144, rounded corners=4pt] (-0.76,1.3)--(-0.4, 1.37)--(0.4,1.37)--(0.76,1.3);  
 \draw[thick, dashed, blue, rotate=-108, rounded corners=4pt] (-0.02,1.8)--(-0.1, 2.15)--(-0.02,2.5);  
 \draw[thick, blue, rotate=-108, rounded corners=4pt]  (0.02,1.8)--(0.1, 2.15)--(0.02,2.5); 
  \node[scale=0.6, blue] at (0,0.9) {$+$};
  \node[scale=0.6, red] at (-0.6,-0.8) {$-$};
  \node[scale=0.6, red] at (0.6,-0.8) {$-$};
  \node[scale=0.6, blue] at (1.4,-0.1) {$+$};
\end{scope}

\begin{scope} [xshift=0cm, yshift=-11.6cm, rotate=-72]
 \draw[very thick, violet] (0,0) circle [radius=2.5cm];
 \draw[very thick, violet, rotate=36] (0,1.6) circle [radius=0.2cm]; 
 \draw[very thick, violet, rotate=108] (0,1.6) circle [radius=0.2cm];  
 \draw[very thick, violet, rotate=-36] (0,1.6) circle [radius=0.2cm]; 
 \draw[very thick, violet, rotate=-108] (0,1.6) circle [radius=0.2cm]; 
 \draw[very thick, violet, rotate=-180] (0,1.6) circle [radius=0.2cm];  
 \draw[thick, red, rotate=-144, rounded corners=4pt] (-0.76,1.27)--(-0.4, 1.2)--(0.4,1.2)--(0.76,1.27);  
 \draw[thick, red, dashed,  rotate=-144, rounded corners=4pt] (-0.76,1.3)--(-0.4, 1.37)--(0.4,1.37)--(0.76,1.3);  
 \draw[thick, blue, rotate=0, rounded corners=4pt] (-0.76,1.27)--(-0.4, 1.2)--(0.4,1.2)--(0.76,1.27);  
 \draw[thick, blue, dashed,  rotate=0, rounded corners=4pt] (-0.76,1.3)--(-0.4, 1.37)--(0.4,1.37)--(0.76,1.3);  
 \draw[thick, red, rotate=144, rounded corners=4pt] (-0.76,1.27)--(-0.4, 1.2)--(0.4,1.2)--(0.76,1.27);  
 \draw[thick, red, dashed,  rotate=144, rounded corners=4pt] (-0.76,1.3)--(-0.4, 1.37)--(0.4,1.37)--(0.76,1.3);  
 \draw[thick, dashed, blue, rotate=-108, rounded corners=4pt] (-0.02,1.8)--(-0.1, 2.15)--(-0.02,2.5);  
 \draw[thick, blue, rotate=-108, rounded corners=4pt]  (0.02,1.8)--(0.1, 2.15)--(0.02,2.5); 
  \node[scale=0.6, blue] at (0,0.9) {$+$};
  \node[scale=0.6, red] at (-0.6,-0.8) {$-$};
  \node[scale=0.6, red] at (0.6,-0.8) {$-$};
  \node[scale=0.6, blue] at (1.4,-0.1) {$+$};
\end{scope}

\begin{scope} [xshift=5.5cm, yshift=-11.6cm, rotate=-72]
 \draw[very thick, violet] (0,0) circle [radius=2.5cm];
 \draw[very thick, violet, rotate=36] (0,1.6) circle [radius=0.2cm]; 
 \draw[very thick, violet, rotate=108] (0,1.6) circle [radius=0.2cm];  
 \draw[very thick, violet, rotate=-36] (0,1.6) circle [radius=0.2cm]; 
 \draw[very thick, violet, rotate=-108] (0,1.6) circle [radius=0.2cm]; 
 \draw[very thick, violet, rotate=-180] (0,1.6) circle [radius=0.2cm];  
 \draw[thick,blue, rotate=-144, rounded corners=4pt] (-0.76,1.27)--(-0.4, 1.2)--(0.4,1.2)--(0.76,1.27);  
 \draw[thick, blue, dashed,  rotate=-144, rounded corners=4pt] (-0.76,1.3)--(-0.4, 1.37)--(0.4,1.37)--(0.76,1.3);  
 \draw[thick, red, rotate=0, rounded corners=4pt] (-0.76,1.27)--(-0.4, 1.2)--(0.4,1.2)--(0.76,1.27);  
 \draw[thick, red, dashed,  rotate=0, rounded corners=4pt] (-0.76,1.3)--(-0.4, 1.37)--(0.4,1.37)--(0.76,1.3);  
 \draw[thick, dashed, red, rotate=-108, rounded corners=4pt] (-0.02,1.8)--(-0.1, 2.15)--(-0.02,2.5);  
 \draw[thick, red, rotate=-108, rounded corners=4pt]  (0.02,1.8)--(0.1, 2.15)--(0.02,2.5); 
 \draw[thick, blue, rotate=108] (0,1.6) circle [radius=0.35cm]; 

  \node[scale=0.6, red] at (0,0.9) {$-$};
  \node[scale=0.6, blue] at (-0.9,-0.3) {$+$};
  \node[scale=0.6, blue] at (0.6,-0.8) {$+$};
  \node[scale=0.6, red] at (1.4,-0.1) {$-$};
\end{scope}

\begin{scope} [xshift=11cm, yshift=-11.6cm, rotate=-72]
 \draw[very thick, violet] (0,0) circle [radius=2.5cm];
 \draw[very thick, violet, rotate=36] (0,1.6) circle [radius=0.2cm]; 
 \draw[very thick, violet, rotate=108] (0,1.6) circle [radius=0.2cm];  
 \draw[very thick, violet, rotate=-36] (0,1.6) circle [radius=0.2cm]; 
 \draw[very thick, violet, rotate=-108] (0,1.6) circle [radius=0.2cm]; 
 \draw[very thick, violet, rotate=-180] (0,1.6) circle [radius=0.2cm];  
 \draw[thick,blue, rotate=-144, rounded corners=4pt] (-0.76,1.27)--(-0.4, 1.2)--(0.4,1.2)--(0.76,1.27);  
 \draw[thick, blue, dashed,  rotate=-144, rounded corners=4pt] (-0.76,1.3)--(-0.4, 1.37)--(0.4,1.37)--(0.76,1.3);  
 \draw[thick, red, rotate=0, rounded corners=4pt] (-0.76,1.27)--(-0.4, 1.2)--(0.4,1.2)--(0.76,1.27);  
 \draw[thick, red, dashed,  rotate=0, rounded corners=4pt] (-0.76,1.3)--(-0.4, 1.37)--(0.4,1.37)--(0.76,1.3);  
 \draw[thick, blue, rotate=72, rounded corners=4pt] (-0.76,1.27)--(-0.4, 1.2)--(0.4,1.2)--(0.76,1.27);  
 \draw[thick,blue, dashed,  rotate=72, rounded corners=4pt] (-0.76,1.3)--(-0.4, 1.37)--(0.4,1.37)--(0.76,1.3);  
 \draw[thick, dashed, red, rotate=-108, rounded corners=4pt] (-0.02,1.8)--(-0.1, 2.15)--(-0.02,2.5);  
 \draw[thick, red, rotate=-108, rounded corners=4pt]  (0.02,1.8)--(0.1, 2.15)--(0.02,2.5); 

%
  \node[scale=0.6, red] at (0,0.9) {$-$};
  \node[scale=0.6, blue] at (-0.8,0.3) {$+$};
  \node[scale=0.6, blue] at (0.6,-0.8) {$+$};
  \node[scale=0.6, red] at (1.4,-0.1) {$-$};
\end{scope}

\node[scale=0.9] at (-2.2,2.1) {$F_1$};
\node[scale=0.9] at (3.3,2.1) {$F_2$};
\node[scale=0.9] at (8.8,2.1) {$F_3$};
\node[scale=0.9] at (-2.2, -3.7) {$F_4$};
\node[scale=0.9] at (3.3,-3.7) {$F_5$};
\node[scale=0.9] at (8.8,-3.7) {$F_6$};
\node[scale=0.9] at (-2.2,-9.5) {$F'_6$};
\node[scale=0.9] at (3.3,-9.5) {$F_7$};
\node[scale=0.9] at (8.8,-9.5) {$F_8$};

\end{tikzpicture}
 \caption{The Dehn twist curves of $F_1, \ldots, F_8$ and $F'_6$ in the proof of Theorem~\ref{prop:g=odd} drawn on $\Sigma_5$.}
\end{figure}

\begin{proof}
First note that the rotation $R$ on $\Sigma_g$ maps each $a_i, b_i, c_i$ to 
$a_{i+1}, b_{i+1}, c_{i+1}$. Here, the indices are considered modulo $g$.

Let $F_1:=A_1A_2C_2^{-1}B_4^{-1}$ and let $\Gamma$ be the subgroup of $\MCG$ generated
by $R$ and  $F_1$. We make the following series of observations:

$F_2:=RF_1R^{-1}=A_2A_3C_3^{-1}B_5^{-1} \in \Gamma$.

$F_3:=(F_2F_1)F_2(F_2F_1)^{-1}=A_2A_3B_4^{-1}B_5^{-1} \in \Gamma$.

\noindent Let us spell out the details of this calculation, as we will have several others akin to this one.
It is easy to see that the diffeomorphism $F_2F_1$ maps 
the curves $a_2, a_3,c_3,b_5$ to $a_2, a_3,b_4,b_5$, respectively, so that
\begin{eqnarray*}
F_3 
& :=& (F_2F_1)F_2(F_2F_1)^{-1} \\
& = & (F_2F_1) (A_2A_3C_3^{-1}B_5^{-1}) (F_2F_1)^{-1} \\
& = & A_2A_3B_4^{-1}B_5^{-1} . 
\end{eqnarray*}


We then have
$F_3^{-1}F_2=B_4C_3^{-1}\in \Gamma$  and  $ F_3F_2^{-1}=B_4^{-1}C_3\in \Gamma.$
By conjugating these elements with powers of $R$, we see that 
\[
B_i C_{i-1}^{-1} \in \Gamma \ \mbox{ and }\ B_i^{-1}C_{i-1}\in \Gamma
\]
 for all $i$. 
 
 We also have 
 \begin{eqnarray*}
 F_4
 &:=& R^{-1}F_3^{-1}R=A_1^{-1}A_2^{-1}B_3B_4 \in \Gamma,\\
 F_5
 &:=& (F_4F_3)F_4(F_4F_3)^{-1}=A_1^{-1}A_2^{-1}A_3B_4 \in \Gamma.
 \end{eqnarray*}
  Thus, 
  $F_5F_4^{-1}=A_3B_3^{-1}\in \Gamma $ and  $ F_5^{-1}F_4=A_3^{-1}B_3\in \Gamma$. 
  Again, 
  by conjugating with powers of $R$, we conclude that 
  \begin{eqnarray} \label{F=aibi}
 A_iB_i^{-1}, A_i^{-1}B_i \in \Gamma, 
 \end{eqnarray}
 and in turn,  
 \[
 A_iC_{i-1}^{-1}\in \Gamma
 \] 
 as well, as we already have $B_i C_{i-1}^{-1} \in \Gamma$.
 
 Furthermore,
 \begin{eqnarray*}
 F_6&:=&(C_g A_1^{-1})F_1(B_4C_3^{-1})=C_gA_2 C_2^{-1}C_3^{-1} \in \Gamma,\\
 F'_6&:=& R^{-4} F_6R^4= C_{g-4} A_{g-2} C_{g-2}^{-1}C_{g-1}^{-1} \in \Gamma,\\
 F_7&:=&(F'_6)^{-1}  (C_{g-1}^{-1}B_g)= C_{g-4}^{-1}  A_{g-2} ^{-1}C_{g-2}B_g  \in \Gamma
 \end{eqnarray*}  and 
  \begin{eqnarray*}
 F_8&:=&(F_7F_6)F_7(F_7F_6)^{-1}= C_{g-4}^{-1}  A_{g-2} ^{-1}C_{g-2}C_g  \in \Gamma, 
\end{eqnarray*} 
 by a similar calculation to the one we had for $F_3$ above.
 From these, we get
 $ F_7F_8^{-1}=B_gC_g^{-1}\in \Gamma$,
  so that 
  \begin{eqnarray} \label{F=bici}
  B_iC_i^{-1}\in \Gamma
  \end{eqnarray}
   by the action of $R$.
 
 \noindent Hence we have all of the following elements in $\Gamma$:
 \begin{eqnarray}
      B_iB_{i+1}^{-1}&=&(B_iC_i^{-1})(C_iB_{i+1}^{-1}), \label{F=bler}\\
     C_iC_{i+1}^{-1}&=&(C_iB_{i+1}^{-1})(B_{i+1}C_{i+1}^{-1}).\label{F=cler}
 \end{eqnarray}

It follows from~\eqref{F=aibi}--\eqref{F=cler} and
Lemma~\ref{lem:G=Mod} that $\Gamma=\mod(\Sigma_g)$.
 \end{proof}

\begin{theorem} \label{prop:g=ufak} 
For $g\geq 3$, the mapping class group $\MCG$ is generated by 
$T, A_1A_2^{-1}$ and $A_1B_1C_1C_2^{-1}B_3^{-1}C_3^{-1}$.
\end{theorem}
\begin{proof}
Let $\Gamma$ be the subgroup of $\mod(\Sigma_g)$ generated by $T, A_1A_2^{-1}$ and 
$F:=A_1B_1C_1C_2^{-1}B_3^{-1}C_3^{-1}$. 

Below we will repeatedly use the conjugation relation, both when conjugating with $F$ and with powers of $T$. The action of $T$ on $\Sigma_g$ maps $a_1$ to $c_1$, $c_i$ to $c_{i+1}$ for each $i = 1, \ldots, g-2$, $c_{g-1}$ to $a_g$, and $a_g$ back to $a_1$, whereas it maps  $b_i$ to $b_{i+1}$ for each $i=1, \ldots g$, and $b_{g+1}$ back to $b_1$.

Note that $F(a_2)=a_2$. Since $F(a_1)=b_1$ and $F(b_1)=c_1$, 
\[
F(A_1A_2^{-1})F^{-1}=B_1A_2^{-1}\in\Gamma
\]
and
\[
F(B_1A_2^{-1})F^{-1}=C_1A_2^{-1}\in\Gamma.
\]
It follows that 
\[
A_1B_1^{-1},B_1C_1^{-1}\in\Gamma. 
\]
Hence, the elements
\begin{eqnarray*}
C_1B_2^{-1}&=&T(A_1B_1^{-1})T^{-1} ,\\
A_2B_2^{-1}&=&(A_2C_1^{-1})(C_1B_2^{-1}) ,\\
B_2C_2^{-1}&=&T(B_1C_1^{-1})T^{-1} ,\\ 
C_1C_2^{-1}&=&(C_1B_2^{-1})(B_2C_2^{-1}) ,\\
B_1B_2^{-1}&=&(B_1C_1^{-1})(C_1B_2^{-1}) , \\
C_jC_{j+1}^{-1}&=&T^{j-1}(C_1C_{2}^{-1})T^{-(j-1)} \mbox{ for } 1\leq j\leq g-2 ,\\
B_iB_{i+1}^{-1}&=&T^{i-1}(B_1B_{2}^{-1})T^{-(i-1)} \mbox{ for } 1\leq i\leq g-1
\end{eqnarray*}
are all in $\Gamma$. It follows now from Lemma~\ref{lem:G=Mod} that 
$\Gamma=\mod(\Sigma_g)$.
\end{proof}


\begin{figure}[h!]	
\centering
\begin{tikzpicture}[every node/.style={inner sep=0pt}, scale=0.9]
\begin{scope} [yshift=0cm, scale=0.4]
 \draw[very thick,rounded corners=10pt] (-7.8,-2) -- (7.8,-2)--(8.75,0)-- (7.8,2) -- (-7.8,2) --(-8.75, 0)-- cycle;
 \draw[very thick, xshift=1.3cm] (0,0) circle [radius=0.6cm];
 \draw[very thick, xshift=3.9cm] (0,0) circle [radius=0.6cm];
 \draw[very thick, xshift=6.5cm] (0,0) circle [radius=0.6cm];
 \draw[very thick, xshift=-1.3cm] (0,0) circle [radius=0.6cm];
 \draw[very thick, xshift=-3.9cm] (0,0) circle [radius=0.6cm];
 \draw[very thick, xshift=-6.5cm] (0,0) circle [radius=0.6cm];
 \draw[thick, red, xshift=1.3cm] (0,0) circle [radius=0.9cm];
 \draw[thick, red, rounded corners=4pt, xshift=-2.6cm] (-0.7,-0.03) -- (-0.4,-0.23)--(0.4,-0.23) --(0.7,-0.03);
 \draw[thick, red, dashed, rounded corners=4pt, xshift=-2.6cm] (-0.7,0.03) -- (-0.4,0.23)--(0.4,0.23) --(0.7,0.03);
\draw[thick, blue, rounded corners=4pt, xshift=-7.8cm] (-0.7,-0.03) -- (-0.4,-0.23)--(0.4,-0.23) --(0.7,-0.03);
 \draw[thick, blue, dashed, rounded corners=4pt, xshift=-7.8cm] (-0.7,0.03) -- (-0.4,0.23)--(0.4,0.23) --(0.7,0.03);
\draw[thick, blue, rounded corners=3pt] (-3.93,0.6) -- (-4.13, 1)--(-4.13,1.6) --(-3.93,2);
 \draw[thick, blue, dashed, rounded corners=3pt] (-3.87,0.6) -- (-3.67, 1)--(-3.67,1.6) --(-3.87,2);
\node[blue, scale=0.6] at (-7.8 ,-0.6) {$+$};
\node[blue, scale=0.6] at (-4.5,1.2) {$+$};
\node[red, scale=0.6] at (-2.6,-0.6) {$-$};
\node[red, scale=0.6] at (1.3,1.2) {$-$};
\end{scope}

\begin{scope} [yshift=-2cm, scale=0.4]
 \draw[very thick,rounded corners=10pt] (-7.8,-2) -- (7.8,-2)--(8.75,0)-- (7.8,2) -- (-7.8,2) --(-8.75, 0)-- cycle;
 \draw[very thick, xshift=1.3cm] (0,0) circle [radius=0.6cm];
 \draw[very thick, xshift=3.9cm] (0,0) circle [radius=0.6cm];
 \draw[very thick, xshift=6.5cm] (0,0) circle [radius=0.6cm];
 \draw[very thick, xshift=-1.3cm] (0,0) circle [radius=0.6cm];
 \draw[very thick, xshift=-3.9cm] (0,0) circle [radius=0.6cm];
 \draw[very thick, xshift=-6.5cm] (0,0) circle [radius=0.6cm];
 \draw[thick, red, xshift=3.9cm] (0,0) circle [radius=0.9cm];
 \draw[thick, red, rounded corners=4pt, xshift=0cm] (-0.7,-0.03) -- (-0.4,-0.23)--(0.4,-0.23) --(0.7,-0.03);
 \draw[thick, red, dashed, rounded corners=4pt, xshift=0cm] (-0.7,0.03) -- (-0.4,0.23)--(0.4,0.23) --(0.7,0.03);
\draw[thick, blue, rounded corners=4pt, xshift=-5.2cm] (-0.7,-0.03) -- (-0.4,-0.23)--(0.4,-0.23) --(0.7,-0.03);
 \draw[thick, blue, dashed, rounded corners=4pt, xshift=-5.2cm] (-0.7,0.03) -- (-0.4,0.23)--(0.4,0.23) --(0.7,0.03);
\draw[thick, blue, rounded corners=10pt]              (-6,-0.3) -- (-5, -1.2)--(-2.8,-1.2) --(-1.8,-0.3);
\draw[thick, blue, dashed, rounded corners=2pt] (-5.9,-0.25) --(-5.7,-0.3)-- (-5, -0.8)--(-3.9,-1)--
(-2.8,-0.8) --(-2.1,-0.2)--(-1.88,-0.2);
\node[blue, scale=0.6] at (-3.8 ,-1.5) {$+$};
\node[blue, scale=0.6] at (-5.2,0.6) {$+$};
\node[red, scale=0.6] at (0,-0.6) {$-$};
\node[red, scale=0.6] at (3.9,1.2) {$-$};
\end{scope}

\begin{scope} [yshift=-4cm, scale=0.4]
 \draw[very thick,rounded corners=10pt] (-7.8,-2) -- (7.8,-2)--(8.75,0)-- (7.8,2) -- (-7.8,2) --(-8.75, 0)-- cycle;
 \draw[very thick, xshift=1.3cm] (0,0) circle [radius=0.6cm];
 \draw[very thick, xshift=3.9cm] (0,0) circle [radius=0.6cm];
 \draw[very thick, xshift=6.5cm] (0,0) circle [radius=0.6cm];
 \draw[very thick, xshift=-1.3cm] (0,0) circle [radius=0.6cm];
 \draw[very thick, xshift=-3.9cm] (0,0) circle [radius=0.6cm];
 \draw[very thick, xshift=-6.5cm] (0,0) circle [radius=0.6cm];
 \draw[thick, red, xshift=1.3cm] (0,0) circle [radius=0.9cm];
 \draw[thick, red, xshift=3.9cm] (0,0) circle [radius=0.9cm];
\draw[thick, blue, rounded corners=4pt, xshift=-5.2cm] (-0.7,-0.03) -- (-0.4,-0.23)--(0.4,-0.23) --(0.7,-0.03);
 \draw[thick, blue, dashed, rounded corners=4pt, xshift=-5.2cm] (-0.7,0.03) -- (-0.4,0.23)--(0.4,0.23) --(0.7,0.03);
\draw[thick, blue, rounded corners=10pt](-6,-0.3) -- (-5, -1.2)--(-2.8,-1.2) --(-1.8,-0.3);
\draw[thick, blue, dashed, rounded corners=2pt] (-5.9,-0.25) --(-5.7,-0.3)-- (-5, -0.8)--(-3.9,-1)--
(-2.8,-0.8) --(-2.1,-0.2)--(-1.88,-0.2);
\node[blue, scale=0.6] at (-3.8 ,-1.5) {$+$};
\node[blue, scale=0.6] at (-5.2,0.6) {$+$};
\node[red, scale=0.6] at (3.9,1.2) {$-$};
\node[red, scale=0.6] at (1.3,1.2) {$-$};
\end{scope}

\begin{scope} [yshift=-6cm, scale=0.4]
 \draw[very thick,rounded corners=10pt] (-7.8,-2) -- (7.8,-2)--(8.75,0)-- (7.8,2) -- (-7.8,2) --(-8.75, 0)-- cycle;
 \draw[very thick, xshift=1.3cm] (0,0) circle [radius=0.6cm];
 \draw[very thick, xshift=3.9cm] (0,0) circle [radius=0.6cm];
 \draw[very thick, xshift=6.5cm] (0,0) circle [radius=0.6cm];
 \draw[very thick, xshift=-1.3cm] (0,0) circle [radius=0.6cm];
 \draw[very thick, xshift=-3.9cm] (0,0) circle [radius=0.6cm];
 \draw[very thick, xshift=-6.5cm] (0,0) circle [radius=0.6cm];
 \draw[thick, blue, xshift=-1.3cm] (0,0) circle [radius=0.9cm];
 \draw[thick, blue, xshift=1.3cm] (0,0) circle [radius=0.9cm];
\draw[thick, red, rounded corners=4pt, xshift=-7.8cm] (-0.7,-0.03) -- (-0.4,-0.23)--(0.4,-0.23) --(0.7,-0.03);
 \draw[thick, red, dashed, rounded corners=4pt, xshift=-7.8cm] (-0.7,0.03) -- (-0.4,0.23)--(0.4,0.23) --(0.7,0.03);
\draw[thick, red, rounded corners=3pt] (-3.93,0.6) -- (-4.13, 1)--(-4.13,1.6) --(-3.93,2);
 \draw[thick, red, dashed, rounded corners=3pt] (-3.87,0.6) -- (-3.67, 1)--(-3.67,1.6) --(-3.87,2);
\node[red, scale=0.6] at (-7.8 ,-0.6) {$-$};
\node[red, scale=0.6] at (-4.5,1.2) {$-$};
\node[blue, scale=0.6] at (-1.3,1.2) {$+$};
\node[blue, scale=0.6] at (1.3,1.2) {$+$};
\end{scope}

\begin{scope} [xshift=7.2cm, yshift=0cm, scale=0.4]
 \draw[very thick,rounded corners=10pt] (-7.8,-2) -- (7.8,-2)--(8.75,0)-- (7.8,2) -- (-7.8,2) --(-8.75, 0)-- cycle;
 \draw[very thick, xshift=1.3cm] (0,0) circle [radius=0.6cm];
 \draw[very thick, xshift=3.9cm] (0,0) circle [radius=0.6cm];
 \draw[very thick, xshift=6.5cm] (0,0) circle [radius=0.6cm];
 \draw[very thick, xshift=-1.3cm] (0,0) circle [radius=0.6cm];
 \draw[very thick, xshift=-3.9cm] (0,0) circle [radius=0.6cm];
 \draw[very thick, xshift=-6.5cm] (0,0) circle [radius=0.6cm];
 \draw[thick, blue, xshift=-1.3cm] (0,0) circle [radius=0.9cm];
 \draw[thick, red, xshift=1.3cm] (0,0) circle [radius=0.9cm];
 \draw[thick, red, xshift=3.9cm] (0,0) circle [radius=0.9cm];
\draw[thick, blue, rounded corners=4pt, xshift=-5.2cm] (-0.7,-0.03) -- (-0.4,-0.23)--(0.4,-0.23) --(0.7,-0.03);
 \draw[thick, blue, dashed, rounded corners=4pt, xshift=-5.2cm] (-0.7,0.03) -- (-0.4,0.23)--(0.4,0.23) --(0.7,0.03);
\node[blue, scale=0.6] at (-5.2,-0.6) {$+$};
\node[blue, scale=0.6] at (-1.3 ,1.2) {$+$};
\node[red, scale=0.6] at (3.9,1.2) {$-$};
\node[red, scale=0.6] at (1.3,1.2) {$-$};
\end{scope}

\begin{scope} [xshift=7.2cm, yshift=-2cm, scale=0.4]
 \draw[very thick,rounded corners=10pt] (-7.8,-2) -- (7.8,-2)--(8.75,0)-- (7.8,2) -- (-7.8,2) --(-8.75, 0)-- cycle;
 \draw[very thick, xshift=1.3cm] (0,0) circle [radius=0.6cm];
 \draw[very thick, xshift=3.9cm] (0,0) circle [radius=0.6cm];
 \draw[very thick, xshift=6.5cm] (0,0) circle [radius=0.6cm];
 \draw[very thick, xshift=-1.3cm] (0,0) circle [radius=0.6cm];
 \draw[very thick, xshift=-3.9cm] (0,0) circle [radius=0.6cm];
 \draw[very thick, xshift=-6.5cm] (0,0) circle [radius=0.6cm];
 \draw[thick, red, xshift=-1.3cm] (0,0) circle [radius=0.9cm];
 \draw[thick, red, xshift=1.3cm] (0,0) circle [radius=0.9cm];
 \draw[thick, blue, xshift=-3.9cm] (0,0) circle [radius=0.9cm];
 \draw[thick, blue, xshift=-6.5cm] (0,0) circle [radius=0.9cm];
\node[blue, scale=0.6] at (-6.5 ,1.2) {$+$};
\node[blue, scale=0.6] at (-3.9,1.2) {$+$};
\node[red, scale=0.6] at (-1.3,1.2) {$-$};
\node[red, scale=0.6] at (1.3,1.2) {$-$};
\end{scope}

\begin{scope} [xshift=7.2cm,yshift=-6cm, scale=0.4]
 \draw[very thick,rounded corners=10pt] (-7.8,-2) -- (7.8,-2)--(8.75,0)-- (7.8,2) -- (-7.8,2) --(-8.75, 0)-- cycle;
 \draw[very thick, xshift=1.3cm] (0,0) circle [radius=0.6cm];
 \draw[very thick, xshift=3.9cm] (0,0) circle [radius=0.6cm];
 \draw[very thick, xshift=6.5cm] (0,0) circle [radius=0.6cm];
 \draw[very thick, xshift=-1.3cm] (0,0) circle [radius=0.6cm];
 \draw[very thick, xshift=-3.9cm] (0,0) circle [radius=0.6cm];
 \draw[very thick, xshift=-6.5cm] (0,0) circle [radius=0.6cm];
 \draw[thick, blue, xshift=6.5cm] (0,0) circle [radius=0.9cm];
 \draw[thick, blue, xshift=3.9cm] (0,0) circle [radius=0.9cm];
 \draw[thick, red, xshift=-6.5cm] (0,0) circle [radius=0.9cm];
 \draw[thick, red, rounded corners=9pt] (-7.3,-1.6) -- (7.3,-1.6)--(8.2,0)-- (7.3,1.6) -- (-7.3,1.6) --(-8.2, 0)-- cycle;
\node[red, scale=0.6] at (-6.5 ,1.2) {$-$};
\node[red, scale=0.6] at (0,1.2) {$-$};
\node[blue, scale=0.6] at (3.9,1.2) {$+$};
\node[blue, scale=0.6] at (6.5,1.2) {$+$};

\end{scope}

\begin{scope} [xshift=7.2cm, yshift=-4cm, scale=0.4]
 \draw[very thick,rounded corners=10pt] (-7.8,-2) -- (7.8,-2)--(8.75,0)-- (7.8,2) -- (-7.8,2) --(-8.75, 0)-- cycle;
 \draw[very thick, xshift=1.3cm] (0,0) circle [radius=0.6cm];
 \draw[very thick, xshift=3.9cm] (0,0) circle [radius=0.6cm];
 \draw[very thick, xshift=6.5cm] (0,0) circle [radius=0.6cm];
 \draw[very thick, xshift=-1.3cm] (0,0) circle [radius=0.6cm];
 \draw[very thick, xshift=-3.9cm] (0,0) circle [radius=0.6cm];
 \draw[very thick, xshift=-6.5cm] (0,0) circle [radius=0.6cm];
  \draw[thick, blue, xshift=6.5cm] (0,0) circle [radius=0.9cm];
 \draw[thick, red, xshift=-6.5cm] (0,0) circle [radius=0.9cm];
 \draw[thick, red, rounded corners=9pt] (-7.3,-1.6) -- (7.3,-1.6)--(8.2,0)-- (7.3,1.6) -- (-7.3,1.6) --(-8.2, 0)-- cycle;
 
\draw[thick, blue, rounded corners=4pt, xshift=2.6cm] (-0.7,-0.03) -- (-0.4,-0.23)--(0.4,-0.23) --(0.7,-0.03);
 \draw[thick, blue, dashed, rounded corners=4pt, xshift=2.6cm] (-0.7,0.03) -- (-0.4,0.23)--(0.4,0.23) --(0.7,0.03);

\node[red, scale=0.6] at (-6.5 ,1.2) {$-$};
\node[red, scale=0.6] at (0,1.2) {$-$};
\node[blue, scale=0.6] at (2.6,-0.6) {$+$};
\node[blue, scale=0.6] at (6.5,1.2) {$+$};

\end{scope}

\node[scale=0.7] at (-3.4, 0.8) {$G_1$};
\node[scale=0.7] at (-3.4, -1.2) {$G_2$};
\node[scale=0.7] at (-3.4, -3.2) {$G_3$};
\node[scale=0.7] at (-3.4, -5.2) {$G_4$};
\node[scale=0.7] at (3.8, 0.8) {$G_5$};
\node[scale=0.7] at (3.8, -1.2) {$G_6$};
\node[scale=0.7] at (3.8, -3.2) {$G_7$};
\node[scale=0.7] at (3.8, -5.2) {$G_8$};

\end{tikzpicture}
\caption{The Dehn twist curves of $G_1, \ldots, G_8$ in the proof of Theorem~\ref{prop:g=even}   drawn on $\Sigma_6$.}
\label{fig:even}
\end{figure}
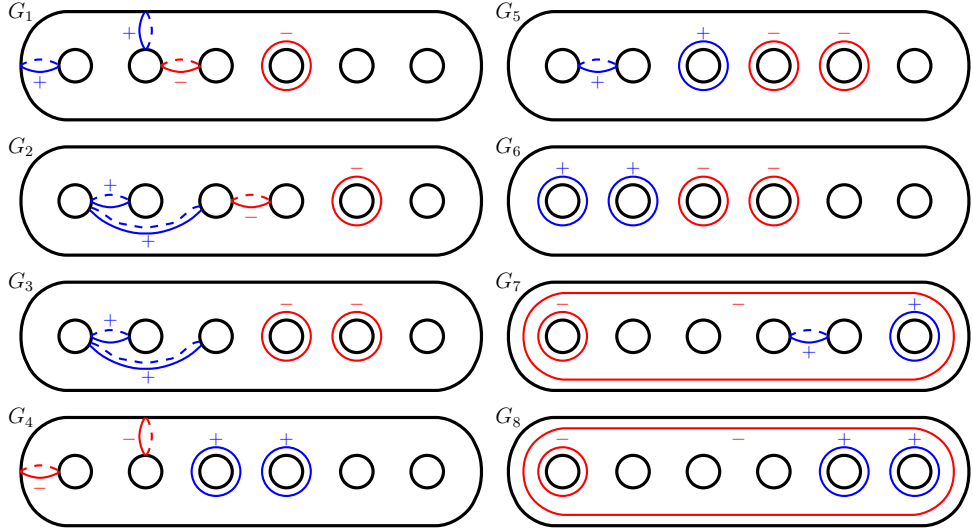

\begin{theorem} \label{prop:g=even}  
For $g\geq 6$, the mapping class group  $\mod(\Sigma_g)$ is generated by  $T$ and $A_1A_2C_2^{-1}B_4^{-1}.$
\end{theorem}
\begin{proof}
Let $\Gamma$ be the subgroup of $\MCG$ generated by $T$ and $G_1:=A_1A_2C_2^{-1}B_4^{-1}$. 
Let $d=T(a_2)$.\footnote{To view the action of $T$ (specifically on the curve $a_2$) in the ``standard'' model as in Figure~\ref{fig:even}, one can observe that $S=(A_1B_1C_1B_2C_2\cdots B_{g-1}C_{g-1}B_gA_g)^2$.}  
We then have
\[
G_2:=TG_1T^{-1}=C_1DC_3^{-1}B_5^{-1} \in \Gamma.
\]
Since $G_2G_1$ maps the curve $c_3$ to $b_4$ and fixes 
 $c_1,d,b_5$, the conjugation of $G_2$ with $G_2G_1$ gives\footnote{Here we use the fact that $a_2$ is disjoint from 
$d=T(a_2)$, which would \emph{not} be the case for $S(a_2)$. This is essentially the reason why we preferred to work with the slightly more complicated torsion element $T$ rather than $S$.}
\begin{eqnarray*}
G_3 & :=& (G_2G_1) G_2(G_2G_1)^{-1}\\
        & = & C_1 D B_4^{-1} B_5^{-1} \in \Gamma.
\end{eqnarray*}

Hence, the subgroup $\Gamma$ contains the elements $G_2G_3^{-1}=C_3^{-1}B_4$ and 
$G_2^{-1}G_3=C_3B_4^{-1}$. Thus, by conjugating by powers of $T$, we see that
\begin{equation}
A_1B_1^{-1}\in \Gamma  \mbox{ and } C_{i}B_{i+1}^{-1}\in \Gamma \label{G=a1b1}
\end{equation}
for $1\leq i \leq g-1$.

We also have
\[
G_4:=G_1^{-1}(C_2^{-1}B_3)=A_1^{-1}A_2^{-1} B_3B_4  \in \Gamma
\]
and
\[
G_5:=(G_3G_4)G_3(G_3 G_4)^{-1}=C_1B_3B_4^{-1}B_5^{-1} \in \Gamma
\]
 by a similar calculation to that of $G_3$ above. From this we get
 \[
 G_3G_5^{-1}=DB_3^{-1} \in \Gamma
 \] 
and hence
\begin{equation}
T^{-1}(DB_3^{-1})T=A_2B_2^{-1}\in \Gamma. \label{G=a2b2}
\end{equation}

Let
\[
G_6:=T^{-1} (B_2C_1^{-1})G_5T= B_1B_2B_3^{-1}B_4^{-1} \in \Gamma
\]
and
\[
G_7:=(T^3 G_5T^{-3})=C_4B_6B_7^{-1}B_8^{-1} \in \Gamma.
\]
Here, we take $B_8=B_1$ if $g=6$. It is then easy to see that
\[
G_8:=(G_7G_6^{-1})G_7(G_7G_6^{-1})^{-1}= B_4B_6B_7^{-1}B_8^{-1} \in \Gamma.
\]
We then have
\[
G_8G_7^{-1}= B_4C_4^{-1} \in \Gamma. 
\]  

 \noindent Thus, by the action of $T$,  we get 
 \begin{equation}
 B_iC_i^{-1}\in \Gamma \label{G=bici}
 \end{equation}
  for all $i$.
As in the proof of Theorem~\ref{prop:g=odd} , we obtain that 
  \begin{equation}
  B_iB_{i+1}^{-1}, C_iC_{i+1}^{-1}\in \Gamma. \label{G=blercler}
  \end{equation}

 Once again from~\eqref{G=a1b1}--\eqref{G=blercler} and  Lemma~\ref{lem:G=Mod}, 
 we conclude  $\Gamma=\mod(\Sigma_g)$.
 \end{proof}


\bigskip
\section{Proof of the main theorem}\label{sec:final}

The proof of our main theorem now follows easily from the array of results we have obtained thus far.

\begin{proof}[Proof of Theorem~\ref{thm:1}] We will obtain our sharpest results in four cases:

\smallskip
\noindent \underline{$g\geq 5$ is odd:} By Theorem~\ref{prop:g=odd}, the mapping class group $\mod(\Sigma_g)$ is generated by $R$ and $A_1A_2C_2^{-1}B_4^{-1}.$ By Proposition~\ref{torsioncomm}, $R$ is a single commutator. On the other hand, there is clearly a diffeomorphism $\phi$ of $\Sigma_g$ mapping the pair
$(a_1,a_2)$ to $(b_4, c_2)$ so that   
\begin{eqnarray*}
A_1A_2C_2^{-1}B_4^{-1}
&=& A_1A_2 (B_4 C_2)^{-1} \\
&=& A_1A_2 (\phi A_1 A_2 \phi^{-1})^{-1} \\
&=& A_1A_2 \phi (A_1 A_2)^{-1} \phi^{-1} \\
&=& [A_1A_2, \phi].
\end{eqnarray*}

\smallskip
\noindent \underline{$g\geq 6$ is even:} By Theorem~\ref{prop:g=even},  the mapping class group $\mod(\Sigma_g)$ is generated by  $T$ and $A_1A_2C_2^{-1}B_4^{-1}$. Again $T$ is a commutator by Proposition~\ref{torsioncomm}, and $A_1A_2C_2^{-1}B_4^{-1}=[A_1 A_2, \phi]$ as above.

\smallskip
\noindent \underline{$g=3$:} By similar arguments we had in Section~\ref{sec:gen}, one can show that when $g \geq 3$, $\mod(\Sigma_g)$ is generated by the elements $R, A_1A_2^{-1},$ and $A_1B_1C_1C_2^{-1}A_3^{-1}B_3^{-1}$; see  Corollary 6 in~\cite{KorkmazInvolutions}. Once again $R$ is a commutator by Proposition~\ref{torsioncomm}. Clearly there is a diffeomorphism $\psi$ of $\Sigma_3$ mapping $a_1$ to $a_2$ and a diffeomorphism $\varphi$ mapping $(a_1, b_1, c_1)$ to $(c_3, b_3, c_2)$. It follows that $ A_1A_2^{-1} = [A_1, \psi]$ and $A_1B_1C_1C_2^{-1}B_3^{-1}C_3^{-1} = [A_1B_1C_1, \varphi]$.

\smallskip
\noindent \underline{$g=4$:}  In this case, by Proposition~\ref{prop:g=ufak},  $\mod(\Sigma_4)$ is generated by the three elements $T, A_1A_2^{-1}$ and $A_1B_1C_1C_2^{-1}B_3^{-1}C_3^{-1}$.
Once again $T$ is a commutator by Proposition~\ref{torsioncomm} and so are the other two elements, as we have argued above.

This completes the proof of Theorem~\ref{thm:1}.
\end{proof}

\bigskip


\begin{thebibliography}{xxxx}



\bibitem{BrendleFarb} T.E. Brendle, B. Farb,
\emph{Every mapping class group is generated by 6 involutions}.
J. of Algebra \textbf{278} (2004), 187--198.

\bibitem{Carmichael} R. D. Carmichael, \emph{Introduction to the theory of groups of finite order}, Dover Publications, New York, 1956. 

\bibitem{Dehn} M. Dehn,
\emph{The group of mapping classes}.  In: Papers on Group Theory and Topology.
Springer-Verlag, 1987. Translated from the German by J. Stillwell 
(Die Gruppe der Abbildungsklassen, Acta Math.  \textbf{69} (1938), 135--206).

\bibitem{EndoKotschick} H. Endo and D. Kotschick, \emph{Bounded cohomology and non-uniform perfection of mapping class groups}. 
Invent. Math. 144 (2001), no. 1, 169–-175. 

\bibitem{FarbMargalit} B. Farb, D. Margalit,
\emph{A primer on mapping class groups}.
Princeton University Press, 2011.

\bibitem{Humphries} S. Humphries,
\emph{Generators for the mapping class group}. 
In: Topology of Low-Dimensional Manifolds, Proc. Second Sussex Conf., Chelwood Gate, 1977, 
Lecture Notes in Math., vol. \textbf{722}, Springer-Verlag, 1979, 44--47.


\bibitem{Korkmaz} M. Korkmaz,
\emph{Generating the surface mapping class group by two elements}.
 Trans. Amer. Math. Soc \textbf{357} (2005), 3299--3310.

\bibitem{KorkmazInvolutions} M. Korkmaz,
\emph{Mapping class group is generated by three involutions.} https://arxiv.org/abs/1904.08156. 

\bibitem{Lickorish} W.B.R. Lickorish,
\emph{A finite set of generators for the homeotopy group of a $2$--manifold}.
 Proc. Cambridge Philos. Soc. \textbf{60} (1964), 769--778.

\bibitem{LiebeckEtAl} M. W. Liebeck, E. A. O'Brien, A. Shalev, P. H. Tiep, \emph{The Ore conjecture.}  J. Eur. Math. Soc. (JEMS) 12 (2010), no. 4, 939-–1008. 

\bibitem{Powell} J. Powell, \emph{Two theorems on the mapping class group of a surface}.
Proc. Amer. Math. Soc. 68 (1978), no. 3, 347–-350. 

\bibitem{Thompson} R. C. Thompson, \emph{Commutators in the special and general linear groups.} Trans. Amer. Math. Soc. 101 (1961), 16–-33.


\bibitem{Wajnryb1996} B. Wajnryb,
\emph{Mapping class group of a surface is generated by two elements}.
Topology  \textbf{35} (1996), 377--383.


\end{thebibliography}
\end{document}